\documentclass[a4paper,UKenglish,cleveref, autoref, thm-restate]{lipics-v2021}
\nolinenumbers 

\pdfoutput=1 
\hideLIPIcs  

\title{The Compositional Structure of Bayesian Inference}

\author{Dylan {Braithwaite}}{Department of Computer and Information Sciences, University of Strathclyde, Scotland}{}{}{}
\author{Jules {Hedges}}{Department of Computer and Information Sciences, University of Strathclyde, Scotland}{}{}{}

\author{Toby {St Clere Smithe}}{
Topos Institute, Berkeley, California
\and
Department of Experimental Psychology, University of Oxford, England}{}{}{}

\authorrunning{D. Braithwaite, J. Hedges, T. {St Clere Smithe}} 

\Copyright{Dylan Braithwaite, Jules Hedges and Toby St Clere Smithe} 



\relatedversion{This paper contains material from unpublished preprints \cite{smithe_bayesian_optically} and \cite{braithwaite_hedges_dependent_bayesian}.
An extended version of this paper may be available on arXiv.org.
} 
\relatedversiondetails[linktext={arXiv:2305.06112}]{Extended Version}{https://arxiv.org/abs/2305.06112}

\EventEditors{J\'{e}r\^{o}me Leroux, Sylvain Lombardy, and David Peleg}
\EventNoEds{3}
\EventLongTitle{48th International Symposium on Mathematical Foundations of Computer Science (MFCS 2023)}
\EventShortTitle{MFCS 2023}
\EventAcronym{MFCS}
\EventYear{2023}
\EventDate{August 28 to September 1, 2023}
\EventLocation{Bordeaux, France}
\EventLogo{}
\SeriesVolume{272}
\ArticleNo{22}

\usepackage{tangle}
\usepackage{graphics}
\usepackage{mathtools}
\usepackage{amsmath}
\usepackage{stmaryrd} 
\usepackage{quiver}

\newtheorem{notation}[definition]{Notation}

\newcommand{\namedSet}[1]{\mathsf{#1}}
\newcommand{\namedCat}[1]{\mathbf{#1}}

\newcommand{\Ca}{\mathcal{C}}
\newcommand{\pr}{\mathbb{P}}
\newcommand{\op}{^\text{op}}
\newcommand{\Ob}{\mathrm{Ob}}

\newcommand{\Cat}{\namedCat{Cat}}
\newcommand{\Fam}{\namedCat{Fam}}
\newcommand{\StFam}{\namedCat{StFam}}
\newcommand{\Set}{\namedCat{Set}}

\newcommand{\BLens}{\namedCat{BLens}}
\newcommand{\DBLens}{\namedCat{DBLens}}
\newcommand{\Stat}{\namedCat{Stat}}
\newcommand{\aeq}{\simeq} 

\newcommand{\dom}{\mathrm{dom}}
\newcommand{\id}{\text{id}}
\newcommand{\inv}{^{-1}}
\newcommand{\defeq}{\coloneqq}

\newcommand{\FinStoch}{\namedCat{FinStoch}}
\newcommand{\dup}{\mathsf{copy}} 
\newcommand{\delete}{\mathsf{delete}}
\newcommand{\ProbStoch}{\mathbf{ProbStoch}}

\newcommand{\restr}[2]{{#1}|_{#2}}
\newcommand{\incl}[1]{\langle{#1}\rangle}

\newcommand{\ctikzfig}[1]{%
\begin{center}
\end{center}}

\begin{document}

\maketitle

\begin{abstract}
Bayes' rule tells us how to invert a causal process in order to update our beliefs in light of new evidence. If the process is believed to have a complex compositional structure, we may observe that the inversion of the whole can be computed piecewise in terms of the component processes. We study the structure of this compositional rule, noting that it relates to the lens pattern in functional programming. 
Working in a suitably general axiomatic presentation of a category of Markov kernels, 
we see how we can think of Bayesian inversion as a particular instance of a state-dependent morphism in a fibred category. We discuss the compositional nature of this, formulated as a functor on the underlying category and explore how this can used for a more type-driven approach to statistical inference.
\end{abstract}

\noindent\\
\emph{This postprint is to appear in the proceedings of the 48th Symposium on Mathematical Foundations of Computer Science (MFCS 2023).}

\section{Introduction}
A \emph{Markov kernel} $f$ is a function whose output is stochastic given the input.
It can be thought of as a stateless device which, if supplied with inputs, will produce an output that depends probabilistically on the input.
Mathematically speaking, this is nothing but a conditional distribution $\pr(y | x)$.
The problem of Bayesian inversion is that we observe the output of such a device but only have a probabilistic `prior' belief about the input, and we would like to update our beliefs about what the input was given the output.
For example, suppose a process takes the roll of an unseen die and tells us whether it was even or odd, except with probability 10\% it flips the result.
If the process tells us that a specific roll resulted in `even' then Bayes' law tells us how we should update a prior belief that the die is uniformly random, to obtain a belief about what the specific roll was.

It is commonly the case that the process is not a black box, but is a composite process formed from simpler pieces.
The conditional distribution formed by sequentially composing two such stochastic functions is given by forming their joint distribution and then marginalising over the middle variable.
In the case of a long chain of sequentially composed functions, belief updating for the whole can be done `compositionally' in terms of belief updating for each individual part, in a process notably similar to backpropagation in neural networks.
A process such as this underlies probabilistic programming languages, which are able to run programs `backwards' after conditioning on some output.

The goal of this paper is to study this process of compositional Bayesian inversion of Markov kernels in isolation, using a suitable axiomatisation of a category of Markov kernels.
We provide a method to build categories whose morphisms are pairs of a Markov kernel and an associated `Bayesian inverter', which is itself built compositionally.

Symmetric monoidal categories with compatible families of copy and delete morphisms have been identified as an expressive language for synthetically representing concepts from probability theory \cite{Cho_Jacobs_2019, Fritz_2020}. 
The typical interpretation given is that the objects represent sets or measurable spaces, and morphisms are Markov kernels between them.
Branded \emph{Markov categories} \cite{Fritz_2020} in this context, these categories have recently seen widespread use in applied category theory as a foundation of probability for two reasons: they allow working axiomatically while abstracting over the specific details of theories such as stochastic matrices, Gaussian kernels and Giry (measure-theoretic) kernels; and they provide a rich string-diagram calculus allowing for simple graphical presentations of many calculations and constructions in probability.
\begin{figure}[h]
\centering	
\begin{tabular}{| c c | c c |}
\hline
$\displaystyle \pr_X(x)\pr_Y(y)$ & \input{diag/parallel.tex} &
$\displaystyle \sum_{y \in Y} \pr_{X \times Y}(-, y)$ &
\input{diag/deletion.tex}
\\
Product distributions & Parallel composition &
Marginalisation & Deletion	\\
\hline
$\displaystyle \pr_X(x)\delta(x, x')$ & \input{diag/copy.tex} &
$\displaystyle \sum_{y \in Y} \pr_f(z |  y) \pr_g(y | x)$ & \input{diag/seq-comp.tex}
\\
Diagonal distributions & Copying &
Chapman-Kolmogorov & Sequential composition
\\
\hline
\end{tabular}
\caption{The graphical representation of common formulas for probability distributions.}
\end{figure}

A categorical translation of Bayes' law allows for a general definition of a \emph{Bayesian inverse} to a morphism in a Markov category.
However, in contrast to many other contexts where we have a notion of dualising morphisms, Bayesian inverses depend on an extra piece of data: the prior distribution.
This is abstracted in our definition as a \emph{state} on an object $X$, or a morphism out of the monoidal unit, $I \to X$, which represents a (non-conditional) probability distribution on $X$.
This leads us to the abstract definition of a Bayesian inverse \cite{Cho_Jacobs_2019}:

\begin{definition}[Bayesian Inversion]
\label{def:binv}
Let $f : X \to Y$ and $p : I \to X$, a kernel $f': Y \to X$ is called a Bayesian inverse of $f$ at $p$ if the following equation holds:
\begin{equation*}
\begin{tangle}{(6,5)}[trim y]
	\tgBlank{(0,0)}{white}
	\tgBlank{(1,0)}{white}
	\tgBlank{(2,0)}{white}
	\tgBlank{(3,0)}{white}
	\tgBlank{(4,0)}{white}
	\tgBlank{(5,0)}{white}
	\tgBlank{(0,1)}{white}
	\tgBlank{(1,1)}{white}
	\tgBlank{(2,1)}{white}
	\tgBorderC{(3,1)}{1}{white}{white}
	\tgBorderA{(4,1)}{white}{white}{white}{white}
	\tgBorder{(4,1)}{0}{1}{0}{1}
	\tgBorderA{(5,1)}{white}{white}{white}{white}
	\tgBorder{(5,1)}{0}{1}{0}{1}
	\tgBorderA{(0,2)}{white}{white}{white}{white}
	\tgBorder{(0,2)}{0}{1}{0}{0}
	\tgBorderA{(1,2)}{white}{white}{white}{white}
	\tgBorder{(1,2)}{0}{1}{0}{1}
	\tgBorderA{(2,2)}{white}{white}{white}{white}
	\tgBorder{(2,2)}{0}{1}{0}{1}
	\tgBorderA{(3,2)}{white}{white}{white}{white}
	\tgBorder{(3,2)}{1}{0}{1}{1}
	\tgBlank{(4,2)}{white}
	\tgBlank{(5,2)}{white}
	\tgBlank{(0,3)}{white}
	\tgBlank{(1,3)}{white}
	\tgBlank{(2,3)}{white}
	\tgBorderC{(3,3)}{0}{white}{white}
	\tgBorderA{(4,3)}{white}{white}{white}{white}
	\tgBorder{(4,3)}{0}{1}{0}{1}
	\tgBorderA{(5,3)}{white}{white}{white}{white}
	\tgBorder{(5,3)}{0}{1}{0}{1}
	\tgBlank{(0,4)}{white}
	\tgBlank{(1,4)}{white}
	\tgBlank{(2,4)}{white}
	\tgBlank{(3,4)}{white}
	\tgBlank{(4,4)}{white}
	\tgBlank{(5,4)}{white}
	\tgCell[(0,1)]{(0,2)}{p}
	\tgCell[(1,1)]{(1.5,2)}{f}
	\tgCell[(1,1)]{(4.5,1)}{f'}
	\tgAxisLabel{(6,1.5)}{west}{X}
	\tgAxisLabel{(6,3.5)}{west}{Y}
\end{tangle}
=
\begin{tangle}{(5,5)}[trim y]
	\tgBlank{(0,0)}{white}
	\tgBlank{(1,0)}{white}
	\tgBlank{(2,0)}{white}
	\tgBlank{(3,0)}{white}
	\tgBlank{(4,0)}{white}
	\tgBlank{(0,1)}{white}
	\tgBlank{(1,1)}{white}
	\tgBorderC{(2,1)}{1}{white}{white}
	\tgBorderA{(3,1)}{white}{white}{white}{white}
	\tgBorder{(3,1)}{0}{1}{0}{1}
	\tgBorderA{(4,1)}{white}{white}{white}{white}
	\tgBorder{(4,1)}{0}{1}{0}{1}
	\tgBlank{(0,2)}{white}
	\tgBorderA{(1,2)}{white}{white}{white}{white}
	\tgBorder{(1,2)}{0}{1}{0}{1}
	\tgBorderA{(2,2)}{white}{white}{white}{white}
	\tgBorder{(2,2)}{1}{0}{1}{1}
	\tgBlank{(3,2)}{white}
	\tgBlank{(4,2)}{white}
	\tgBlank{(0,3)}{white}
	\tgBlank{(1,3)}{white}
	\tgBorderC{(2,3)}{0}{white}{white}
	\tgBorderA{(3,3)}{white}{white}{white}{white}
	\tgBorder{(3,3)}{0}{1}{0}{1}
	\tgBorderA{(4,3)}{white}{white}{white}{white}
	\tgBorder{(4,3)}{0}{1}{0}{1}
	\tgBlank{(0,4)}{white}
	\tgBlank{(1,4)}{white}
	\tgBlank{(2,4)}{white}
	\tgBlank{(3,4)}{white}
	\tgBlank{(4,4)}{white}
	\tgCell[(1,1)]{(3.5,3)}{f}
	\tgCell[(0,1)]{(0.5,2)}{p}
	\tgAxisLabel{(5,1.5)}{west}{X}
	\tgAxisLabel{(5,3.5)}{west}{Y}
\end{tangle}
\end{equation*}
\end{definition}

This categorical based approach to Bayes' law naturally leads to a question of whether Bayesian inverses compose;
that is, can we construct a Bayesian inverse of a composite kernel $f\circ g$ as a composite of the separate Bayesian inverses of $f$ and $g$?
Inspecting the relevant equation
\begin{equation} \label{eq:bayes-comp}
\begin{tangle}{(10,5)}[trim y]
	\tgBlank{(0,0)}{white}
	\tgBlank{(1,0)}{white}
	\tgBlank{(2,0)}{white}
	\tgBlank{(3,0)}{white}
	\tgBlank{(4,0)}{white}
	\tgBlank{(5,0)}{white}
	\tgBlank{(6,0)}{white}
	\tgBlank{(7,0)}{white}
	\tgBlank{(8,0)}{white}
	\tgBlank{(9,0)}{white}
	\tgBlank{(0,1)}{white}
	\tgBlank{(1,1)}{white}
	\tgBlank{(2,1)}{white}
	\tgBlank{(3,1)}{white}
	\tgBlank{(4,1)}{white}
	\tgBorderC{(5,1)}{1}{white}{white}
	\tgBorderA{(6,1)}{white}{white}{white}{white}
	\tgBorder{(6,1)}{0}{1}{0}{1}
	\tgBorderA{(7,1)}{white}{white}{white}{white}
	\tgBorder{(7,1)}{0}{1}{0}{1}
	\tgBorderA{(8,1)}{white}{white}{white}{white}
	\tgBorder{(8,1)}{0}{1}{0}{1}
	\tgBorderA{(9,1)}{white}{white}{white}{white}
	\tgBorder{(9,1)}{0}{1}{0}{1}
	\tgBorderA{(0,2)}{white}{white}{white}{white}
	\tgBorder{(0,2)}{0}{1}{0}{0}
	\tgBorderA{(1,2)}{white}{white}{white}{white}
	\tgBorder{(1,2)}{0}{1}{0}{1}
	\tgBorderA{(2,2)}{white}{white}{white}{white}
	\tgBorder{(2,2)}{0}{1}{0}{1}
	\tgBorderA{(3,2)}{white}{white}{white}{white}
	\tgBorder{(3,2)}{0}{1}{0}{1}
	\tgBorderA{(4,2)}{white}{white}{white}{white}
	\tgBorder{(4,2)}{0}{1}{0}{1}
	\tgBorderA{(5,2)}{white}{white}{white}{white}
	\tgBorder{(5,2)}{1}{0}{1}{1}
	\tgBlank{(6,2)}{white}
	\tgBlank{(7,2)}{white}
	\tgBlank{(8,2)}{white}
	\tgBlank{(9,2)}{white}
	\tgBlank{(0,3)}{white}
	\tgBlank{(1,3)}{white}
	\tgBlank{(2,3)}{white}
	\tgBlank{(3,3)}{white}
	\tgBlank{(4,3)}{white}
	\tgBorderC{(5,3)}{0}{white}{white}
	\tgBorderA{(6,3)}{white}{white}{white}{white}
	\tgBorder{(6,3)}{0}{1}{0}{1}
	\tgBorderA{(7,3)}{white}{white}{white}{white}
	\tgBorder{(7,3)}{0}{1}{0}{1}
	\tgBorderA{(8,3)}{white}{white}{white}{white}
	\tgBorder{(8,3)}{0}{1}{0}{1}
	\tgBorderA{(9,3)}{white}{white}{white}{white}
	\tgBorder{(9,3)}{0}{1}{0}{1}
	\tgBlank{(0,4)}{white}
	\tgBlank{(1,4)}{white}
	\tgBlank{(2,4)}{white}
	\tgBlank{(3,4)}{white}
	\tgBlank{(4,4)}{white}
	\tgBlank{(5,4)}{white}
	\tgBlank{(6,4)}{white}
	\tgBlank{(7,4)}{white}
	\tgBlank{(8,4)}{white}
	\tgBlank{(9,4)}{white}
	\tgCell[(0,1)]{(0,2)}{p}
	\tgCell[(1,1)]{(1.5,2)}{g}
	\tgCell[(1,1)]{(3.5,2)}{f}
	\tgCell[(1,1)]{(6.5,1)}{f'}
	\tgCell[(1,1)]{(8.5,1)}{g'}
\end{tangle}
=
\begin{tangle}{(7,5)}[trim y]
	\tgBlank{(0,0)}{white}
	\tgBlank{(1,0)}{white}
	\tgBlank{(2,0)}{white}
	\tgBlank{(3,0)}{white}
	\tgBlank{(4,0)}{white}
	\tgBlank{(5,0)}{white}
	\tgBlank{(6,0)}{white}
	\tgBlank{(0,1)}{white}
	\tgBlank{(1,1)}{white}
	\tgBorderC{(2,1)}{1}{white}{white}
	\tgBorderA{(3,1)}{white}{white}{white}{white}
	\tgBorder{(3,1)}{0}{1}{0}{1}
	\tgBorderA{(4,1)}{white}{white}{white}{white}
	\tgBorder{(4,1)}{0}{1}{0}{1}
	\tgBorderA{(5,1)}{white}{white}{white}{white}
	\tgBorder{(5,1)}{0}{1}{0}{1}
	\tgBorderA{(6,1)}{white}{white}{white}{white}
	\tgBorder{(6,1)}{0}{1}{0}{1}
	\tgBlank{(0,2)}{white}
	\tgBorderA{(1,2)}{white}{white}{white}{white}
	\tgBorder{(1,2)}{0}{1}{0}{1}
	\tgBorderA{(2,2)}{white}{white}{white}{white}
	\tgBorder{(2,2)}{1}{0}{1}{1}
	\tgBlank{(3,2)}{white}
	\tgBlank{(4,2)}{white}
	\tgBlank{(5,2)}{white}
	\tgBlank{(6,2)}{white}
	\tgBlank{(0,3)}{white}
	\tgBlank{(1,3)}{white}
	\tgBorderC{(2,3)}{0}{white}{white}
	\tgBorderA{(3,3)}{white}{white}{white}{white}
	\tgBorder{(3,3)}{0}{1}{0}{1}
	\tgBorderA{(4,3)}{white}{white}{white}{white}
	\tgBorder{(4,3)}{0}{1}{0}{1}
	\tgBorderA{(5,3)}{white}{white}{white}{white}
	\tgBorder{(5,3)}{0}{1}{0}{1}
	\tgBorderA{(6,3)}{white}{white}{white}{white}
	\tgBorder{(6,3)}{0}{1}{0}{1}
	\tgBlank{(0,4)}{white}
	\tgBlank{(1,4)}{white}
	\tgBlank{(2,4)}{white}
	\tgBlank{(3,4)}{white}
	\tgBlank{(4,4)}{white}
	\tgBlank{(5,4)}{white}
	\tgBlank{(6,4)}{white}
	\tgCell[(1,1)]{(3.5,3)}{f}
	\tgCell[(1,1)]{(5.5,3)}{g}
	\tgCell[(0,1)]{(0.5,2)}{p}
\end{tangle}	
\end{equation}
we can observe that this is indeed true if $f'$ and $g'$ are Bayesian inverse to $f$ and $g$, but we have to ask for $f'$ not to be inverse to $f$ at $p$, but rather at $g\circ p$:
in some sense the type of this composition is dependent on the kernels that it is being applied to.
In other words, assuming a hypothetical function $\mathsf{BayesInv}(f, p)$ which computes Bayesian inverses, the composite could be expressed as
\[\mathsf{BayesInv}(f\circ g, p) = \mathsf{BayesInv}(g, p) \circ \mathsf{BayesInv}(f, g \circ p)\]
Notably this has a similar form to the reverse-mode chain rule for Jacobian matrices that underlies backpropagation, $J^\top_{g \circ f} (x) = J^\top_g (f (x)) J^\top_f (x)$.
As such we think of the composition rule as a \emph{chain rule for Bayesian updating}.
To formalise the compositionality of such a construction we would typically like to promote it to a functor on our Markov category, which we denote in general as $\Ca$. The fact that Bayesian inverses are indexed by $p$ in this way makes choosing the target of such a a Bayesian inversion functor nontrivial however.

One approach to making Bayesian inversion functorial, taken by Cho and Jacobs \cite{Cho_Jacobs_2019}, is to work instead in a category where the objects $X$ are equipped with a choice of $p : I \to X$: in this case, there is always a canonical choice of inverse morphism, making Bayesian inversion into a `dagger' functor.
Although such a setting is still a Markov category, this approach does depart from the operational interpretation of morphisms as stochastic kernels, because the distribution produced is already present in the type of the morphism. So the morphisms have more of a `relational' role.

In this paper we take an alternative approach and instead consider a category where the morphisms are indexed families of kernels.
Consequently, we can represent the entire $(I \to X)$-indexed family of Bayesian inverses as a single morphism, resolving the problems faced previously.
The morphisms of this category have a very similar structure to a \textit{lens} from database theory and functional programming \cite{Foster_Greenwald_Moore_Pierce_Schmitt_2007}, but, instead of performing deterministic updates to data structures, this performs Bayesian updates to beliefs.
We therefore call such pairs \textit{Bayesian lenses}.

This category is however `too big' in the sense that it contains everything whose type matches that of the Bayesian inverse. Rather than a failure of the abstraction, we view this as a feature allowing us to encode not only exact Bayesian inversion, but also approximate updaters and other structures. To  pick out the specific lenses corresponding to the actual Bayesian inverse of $f$ we therefore use a functor $\Ca \to \BLens$, which encodes a choice of Bayesian inverse for each kernel.

The category of Bayesian lenses is constructed as a fibred category that is closely related to the families fibration, commonly used in the semantics of dependent types.
We embrace this relationship by extending $\BLens$ into a larger category constructed from a generalised families fibration which leads us to a category of \emph{dependent Bayesian lenses}, which have not only indexed families as morphisms, but indexed objects as well.
This construction admits Bayesian inverses whose domains are allowed to depend on the prior distribution at which the inverse is taken.
In certain Markov categories this allows us to consider inverses as being restricted to the support of the prior, which puts the construction on a neater theoretical foundation, and which clarifies thinking about the Bayesian inversion of structure maps in the category.
Since Bayes' law does not define how beliefs should be updated after a zero-probability observation, this falls into the usual pattern in computer science of using a type system to `make illegal states unrepresentable'\footnote{This phrase originated with Yaron Minsky in the blog post \url{https://blog.janestreet.com/effective-ml-revisited/}}.

\section{Preliminaries}
\begin{definition}[Markov Categories \cite{Fritz_2020}]
	A Markov category is a symmetric monoidal category $\mathcal C$ with the structure of a commutative comonoid on each object, such that:
	\begin{itemize}
		\item the assignment of counits (`delete' or `discard' maps) to objects is natural; and
		\item comultiplications (`copy' maps) are compatible with the monoidal structure:
			\begin{equation*}
\begin{tangle}{(3,3)}
	\tgBlank{(0,0)}{white}
	\tgBorderC{(1,0)}{1}{white}{white}
	\tgBorderA{(2,0)}{white}{white}{white}{white}
	\tgBorder{(2,0)}{0}{1}{0}{1}
	\tgBorderA{(0,1)}{white}{white}{white}{white}
	\tgBorder{(0,1)}{0}{1}{0}{1}
	\tgBorderA{(1,1)}{white}{white}{white}{white}
	\tgBorder{(1,1)}{1}{0}{1}{1}
	\tgBlank{(2,1)}{white}
	\tgBlank{(0,2)}{white}
	\tgBorderC{(1,2)}{0}{white}{white}
	\tgBorderA{(2,2)}{white}{white}{white}{white}
	\tgBorder{(2,2)}{0}{1}{0}{1}
	\tgAxisLabel{(3,0.5)}{west}{X \otimes Y}
	\tgAxisLabel{(0,1.5)}{east}{X\otimes Y}
	\tgAxisLabel{(3,2.5)}{west}{X \otimes Y}
\end{tangle}
\ \ \ \ \ \ 
			=
\ \ \ 
\begin{tangle}{(5,5)}
	\tgBlank{(0,0)}{white}
	\tgBorderC{(1,0)}{1}{white}{white}
	\tgBorderA{(2,0)}{white}{white}{white}{white}
	\tgBorder{(2,0)}{0}{1}{0}{1}
	\tgBorderA{(3,0)}{white}{white}{white}{white}
	\tgBorder{(3,0)}{0}{1}{0}{1}
	\tgBorderA{(4,0)}{white}{white}{white}{white}
	\tgBorder{(4,0)}{0}{1}{0}{1}
	\tgBorderA{(0,1)}{white}{white}{white}{white}
	\tgBorder{(0,1)}{0}{1}{0}{1}
	\tgBorderA{(1,1)}{white}{white}{white}{white}
	\tgBorder{(1,1)}{1}{0}{1}{1}
	\tgBorderC{(2,1)}{1}{white}{white}
	\tgBorderA{(3,1)}{white}{white}{white}{white}
	\tgBorder{(3,1)}{0}{1}{0}{1}
	\tgBorderA{(4,1)}{white}{white}{white}{white}
	\tgBorder{(4,1)}{0}{1}{0}{1}
	\tgBlank{(0,2)}{white}
	\tgBorderC{(1,2)}{0}{white}{white}
	\tgBorderA{(2,2)}{white}{white}{white}{white}
	\tgBorder{(2,2)}{1}{1}{1}{1}
	\tgBorderC{(3,2)}{2}{white}{white}
	\tgBlank{(4,2)}{white}
	\tgBorderA{(0,3)}{white}{white}{white}{white}
	\tgBorder{(0,3)}{0}{1}{0}{1}
	\tgBorderA{(1,3)}{white}{white}{white}{white}
	\tgBorder{(1,3)}{0}{1}{0}{1}
	\tgBorderA{(2,3)}{white}{white}{white}{white}
	\tgBorder{(2,3)}{1}{0}{1}{1}
	\tgBorderC{(3,3)}{0}{white}{white}
	\tgBorderA{(4,3)}{white}{white}{white}{white}
	\tgBorder{(4,3)}{0}{1}{0}{1}
	\tgBlank{(0,4)}{white}
	\tgBlank{(1,4)}{white}
	\tgBorderC{(2,4)}{0}{white}{white}
	\tgBorderA{(3,4)}{white}{white}{white}{white}
	\tgBorder{(3,4)}{0}{1}{0}{1}
	\tgBorderA{(4,4)}{white}{white}{white}{white}
	\tgBorder{(4,4)}{0}{1}{0}{1}
	\tgAxisLabel{(5,0.5)}{west}{X}
	\tgAxisLabel{(0,1.5)}{east}{X}
	\tgAxisLabel{(5,1.5)}{west}{Y}
	\tgAxisLabel{(0,3.5)}{east}{Y}
	\tgAxisLabel{(5,3.5)}{west}{X}
	\tgAxisLabel{(5,4.5)}{west}{Y}
\end{tangle}
			\end{equation*}
	\end{itemize}
\end{definition}
\begin{example}[Stochastic Matrices, \cite{Fritz_2020} example 2.5]
\label{ex:finstoch}
There is a Markov category $\mathbf{FinStoch}$ whose objects are finite sets, and whose morphisms $f : X \to Y$ are stochastic matrices $f : X \times Y \to [0, 1]$, i.e. satisfying $\sum_{y \in Y} f (x, y) = 1$ for all $x \in X$.
Identity morphisms are identity matrices, and composition of morphisms is by matrix multiplication, known in this context as the Chapman-Kolmogorov equation, $(g \circ f) (x, z) = \sum_{y \in Y} f (x, y) \cdot g (y, z)$.
The symmetric monoidal structure of $\mathbf{FinStoch}$ is given on objects by cartesian product, and on morphisms by tensor product of stochastic matrices: $(f \otimes g) ((x, x'), (y, y')) = f (x, y) \cdot g (x', y')$.
\end{example}

\begin{example}[Gaussian Kernels, \cite{Fritz_2020} section 6]
\label{ex:gauss}
There is a Markov category $\mathbf{Gauss}$ whose objects are Euclidean spaces $\mathbb R^n$, and whose morphisms $f : \mathbb R^m \to \mathbb R^n$ are triples of a matrix $M \in \mathbb R^{m \times n}$, a vector $\mu \in \mathbb R^n$ and a positive semidefinite matrix $\sigma \in \mathbb R^{n \times n}$, considered to represent an affine function with independent Gaussian noise $f (v) = Mv + \mathcal N (\mu, \sigma)$.
The definitions of categorical composition and monoidal product are slightly involved, but can be derived from laws of Gaussian probability.
\end{example}

\begin{example}[Probability Kernels, \cite{Fritz_2020} section 4]
Giry \cite{giry82} introduced a monad $\mathcal G$ on the category of measurable spaces, now known as the Giry monad, taking each measurable space to the space of probability measures on it, equipped with an appropriate measurable structure.
The Kleisli category $\mathrm{Kl} (\mathcal G)$ is a Markov category, also known as $\mathbf{Stoch}$, which allows working with arbitrary measure-theoretic probability and contains many other important Markov categories as subcategories, including the previous two examples.
Although this category is in a sense canonical, it suffers from many undesirable properties due to extremely general nature of measure theory. One example which is relevant in the later sections is the spaces representing the support of certain distributions.
\end{example}

Keeping these examples in mind, we refer to the morphisms of such a category as \emph{Markov kernels} or more simply as \emph{kernels}.

\begin{definition}[Almost-Sure Equality \cite{Cho_Jacobs_2019}]
A pair of kernels $f, g : X \to Y$ are $p$-almost-surely equal for some $p: I \to X$, if the following equality holds
\begin{equation*}
\begin{tangle}{(5,5)}[trim y]
	\tgBlank{(0,0)}{white}
	\tgBlank{(1,0)}{white}
	\tgBlank{(2,0)}{white}
	\tgBlank{(3,0)}{white}
	\tgBlank{(4,0)}{white}
	\tgBlank{(0,1)}{white}
	\tgBlank{(1,1)}{white}
	\tgBorderC{(2,1)}{1}{white}{white}
	\tgBorderA{(3,1)}{white}{white}{white}{white}
	\tgBorder{(3,1)}{0}{1}{0}{1}
	\tgBorderA{(4,1)}{white}{white}{white}{white}
	\tgBorder{(4,1)}{0}{1}{0}{1}
	\tgBlank{(0,2)}{white}
	\tgBorderA{(1,2)}{white}{white}{white}{white}
	\tgBorder{(1,2)}{0}{1}{0}{1}
	\tgBorderA{(2,2)}{white}{white}{white}{white}
	\tgBorder{(2,2)}{1}{0}{1}{1}
	\tgBlank{(3,2)}{white}
	\tgBlank{(4,2)}{white}
	\tgBlank{(0,3)}{white}
	\tgBlank{(1,3)}{white}
	\tgBorderC{(2,3)}{0}{white}{white}
	\tgBorderA{(3,3)}{white}{white}{white}{white}
	\tgBorder{(3,3)}{0}{1}{0}{1}
	\tgBorderA{(4,3)}{white}{white}{white}{white}
	\tgBorder{(4,3)}{0}{1}{0}{1}
	\tgBlank{(0,4)}{white}
	\tgBlank{(1,4)}{white}
	\tgBlank{(2,4)}{white}
	\tgBlank{(3,4)}{white}
	\tgBlank{(4,4)}{white}
	\tgCell[(0,1)]{(0.5,2)}{p}
	\tgCell[(1,1)]{(3.5,1)}{f}
	\tgAxisLabel{(5,1.5)}{west}{Y}
	\tgAxisLabel{(5,3.5)}{west}{X}
\end{tangle}\ \ \ \ \ 
=
\begin{tangle}{(5,5)}[trim y]
	\tgBlank{(0,0)}{white}
	\tgBlank{(1,0)}{white}
	\tgBlank{(2,0)}{white}
	\tgBlank{(3,0)}{white}
	\tgBlank{(4,0)}{white}
	\tgBlank{(0,1)}{white}
	\tgBlank{(1,1)}{white}
	\tgBorderC{(2,1)}{1}{white}{white}
	\tgBorderA{(3,1)}{white}{white}{white}{white}
	\tgBorder{(3,1)}{0}{1}{0}{1}
	\tgBorderA{(4,1)}{white}{white}{white}{white}
	\tgBorder{(4,1)}{0}{1}{0}{1}
	\tgBlank{(0,2)}{white}
	\tgBorderA{(1,2)}{white}{white}{white}{white}
	\tgBorder{(1,2)}{0}{1}{0}{1}
	\tgBorderA{(2,2)}{white}{white}{white}{white}
	\tgBorder{(2,2)}{1}{0}{1}{1}
	\tgBlank{(3,2)}{white}
	\tgBlank{(4,2)}{white}
	\tgBlank{(0,3)}{white}
	\tgBlank{(1,3)}{white}
	\tgBorderC{(2,3)}{0}{white}{white}
	\tgBorderA{(3,3)}{white}{white}{white}{white}
	\tgBorder{(3,3)}{0}{1}{0}{1}
	\tgBorderA{(4,3)}{white}{white}{white}{white}
	\tgBorder{(4,3)}{0}{1}{0}{1}
	\tgBlank{(0,4)}{white}
	\tgBlank{(1,4)}{white}
	\tgBlank{(2,4)}{white}
	\tgBlank{(3,4)}{white}
	\tgBlank{(4,4)}{white}
	\tgCell[(0,1)]{(0.5,2)}{p}
	\tgCell[(1,1)]{(3.5,1)}{g}
	\tgAxisLabel{(5,1.5)}{west}{Y}
	\tgAxisLabel{(5,3.5)}{west}{X}
\end{tangle}
\end{equation*}
In this case we write $f \aeq_p g$.
\end{definition}

We note that this equation is similar in shape to that which defines a Bayesian inverse.
Indeed this pattern of considering a kernel composed onto one branch of a copy is a standard way to consider a kernel in the context of a distribution on its domain.
The appearance of this pattern in the definition of the Bayesian inverse means exactly that we are defining the concept only up to almost-sure equality.

\begin{proposition}[Bayesian inverses are almost-surely equal]
For any kernel $f: X \to Y$ and state $p : I \to X$, if two kernels $g, g' : Y \to X$ both satisfy the conditions of a Bayesian inverse to $f$, then $g \aeq_{f\circ p} g'$.
\qed
\end{proposition}

Considering the interpretation of these equations, this result makes intuitive sense.
It essentially says that Bayesian inversion is uniquely defined only on the points where Bayes' law would not have you divide by zero.
This allows us to be mindful of the partiality of Bayes' law without requiring the added complication of considering partial maps.
However, this will complicate matters later when we establish the functoriality of Bayes' law.
In fact, as we will see, we will need some extra coherence assumptions to ensure that Bayesian inversion is functorial, rather than only almost-surely functorial.

\section{Bayesian Updates Compose Optically}
The Bayesian inverse of a kernel is defined with respect to a prior state on the domain of the kernel.
We want to consider the general inverse of a kernel $A \to B$ as a function which assigns to each prior $p: I \to A$, a $p$-inverse $B \to A$.
However the space of priors with respect to which a kernel can be inverted depends on the domain of the kernel: it is the set of states $\Ca(I, \dom(f))$.
To formalise this dependence we construct an \emph{indexed category}.
Just as indexed sets $(X_i)_{i \in I}$ can be thought of as functions $I \to \namedSet{Set}$, we define indexed categories as (pseudo)functors into $\Cat$.
We consider Bayesian inverses in the context of an indexed category
$\Ca \to \Cat$
sending $X$ to a category of $\Ca(I, X)$-indexed kernels.

\begin{definition}[The $\Stat$ construction]
\label{def:stat-cat}
We define the indexed category of $\Ca$-state-indexed kernels, $\Stat: \Ca\op \to \Cat$, as follows.
\\\\
For each $X$, $\Stat(X)$ is the category of $\Ca(I, X)$-indexed kernels, where
\begin{itemize}
	\item objects are the objects of $\Ca$;
	\item morphisms $A \to B$ are functions $\Ca(I, X) \to \Ca(A, B)$; and
	\item composition and identities are pointwise given by the corresponding structure in $\Ca$.
\end{itemize}
For $f : X \to Y$ in $\Ca$ we obtain a reindexing functor $f^*: \Stat(Y) \to \Stat(X)$ which
\begin{itemize}
	\item acts as the identity on objects; and
	\item reindexes functions by sending $\sigma: \Ca(I, Y) \to \Ca(A, B)$ to $f^*\sigma$ defined by:
\[\begin{tikzcd}[ampersand replacement=\&,row sep=tiny]
	{\Ca(I, X)} \& {\Ca(I, Y)} \& {\Ca(A, B)} \\
	p \& {f\circ p} \& {\sigma(f\circ p)}
	\arrow[from=1-1, to=1-2]
	\arrow[from=1-2, to=1-3]
	\arrow[shorten <=6pt, shorten >=6pt, maps to, from=2-1, to=2-2]
	\arrow[shorten <=5pt, shorten >=5pt, maps to, from=2-2, to=2-3]
\end{tikzcd}\]

\end{itemize}
Since the reindexing functors are defined by pre-composition it is easy to verify that $\Stat$ is indeed a functor.
\end{definition}

This provides sufficient expressive power to represent general Bayesian inverses. For a kernel $f: X \to Y$, the general Bayesian inverse of $f$ is a morphism $Y \to X$ in $\Stat(X)$.
However it is awkward to have to work across a large collection of categories in order to represent the inverses for every kernel. 
Really we would like to assemble the collection of categories into a single category containing all of the inverses.
Fortunately this category is described exactly by the \emph{Grothendieck construction} \cite[\S8.3]{Borceux1994Handbook2}.
The Grothendieck construction realises an equivalence between indexed categories and \emph{fibrations} -- functors with a property of being suitably projection-like.
The core of the construction is a disjoint union of the constituent categories, equipped with morphisms induced by the reindexing functors that connect each of the otherwise disjoint components.
\begin{center}
\begin{tabular}{l l}
	Indexed categories & $F: \Ca \to \Cat$\\
	\hline
	Fibrations & $P_F : \left(\coprod_{C \in \Ca} F(C)\right) \to \Ca$ 
\end{tabular}
\end{center}
We proceed to construct the category of all state-indexed kernels in this way, which we will call the category of \emph{Bayesian lenses}.

\begin{definition}[Bayesian Lenses]
\label{def:blens}
We define the category of \emph{Bayesian lenses} as the Grothendieck construction
$\BLens(\Ca) = \coprod_{X \in \Ca} \Stat(X)\op.$
\end{definition}

Explicitly this is a category whose objects are pairs $\binom X A$ of objects from $\Ca$, and whose morphisms are of the form $(f, f^\sharp): \binom X A \to \binom Y B$ where
$f$ is a kernel $X \to Y$,
and $f^\sharp$ is a family of `backward' kernels, $\Ca(I, X) \to \Ca(B, A)$.
The composition of two such morphisms
	\[\binom X A \xrightarrow{(f, {f^\sharp})} 
	  \binom Y B \xrightarrow{(g, {g^\sharp})} 
	  \binom Z C\]
	is 
	  $(g \circ f, (g \circ f)^\sharp)$
	where 
	  $(g \circ f)^\sharp$
	is a function
	  $\Ca(I, X) \to \Ca(C, A)$, 
	sending $p : I \to X$ to the composition
	\[ C \xrightarrow{g^\sharp(f\circ p)}
	   B \xrightarrow{\ \ f^\sharp(p) \ }
	   A.
	\]
We call the morphisms of this category \emph{Bayesian lenses}.

\begin{remark}
Contrary to our general description of the Grothendieck construction, we in fact used the `fibrewise' opposite of $\Stat$, taking $\Stat(X)\op$ for each $X$ in $\Ca$.
Ultimately this does not change the space of functions we can represent, but it provides a neater type signature for the morphisms we care about.
Namely this means the Bayesian inverse of a kernel will be a morphism $\binom X X \to \binom Y Y$, rather than $\binom X Y \to \binom Y X$.
We think of the appearance of the opposite here as signalling that we are working with a kind of bidirectional process.

Spivak \cite{Spivak_2019} proposes that the fibrewise opposite of a fibration be considered as a generalisation of the lens construction from database theory and functional programming \cite{Foster_Greenwald_Moore_Pierce_Schmitt_2007}.
More clearly, if $\Ca$ is the category of sets (which is degenerately a Markov category), then $f^\sharp : \Ca (I, X) \to \Ca (B, A)$ can be equivalently written $f^\sharp : X \times B \to A$, and we recover the standard definition of a lens.
This motivates our use of the term ``Bayesian lens'' for these morphisms.
\end{remark}

\begin{proposition}[Bayesian inversion is almost functorial]
\label{thm:binv-functorial}
If $\Ca$ has Bayesian inverses for ever kernel at every prior,
then Bayesian inversion defines a functor $T : \Ca \to \BLens(\Ca)$ up to almost-sure equality.
\end{proposition}
\begin{proof}
Because Bayesian inverses are only unique up to almost-equality, $T$ is only almost surely functorial: it maps each object $X$ to $\binom X X$, and each kernel $f:X\to Y$ to a lens $(f,f^\sharp):\binom X X\to\binom Y Y$ whose inverse component is given by a representative of the almost-surely unique family of Bayesian inversions of $f$.

Given a pair of composable kernels $X \xrightarrow{f} Y \xrightarrow{g} Z$, we need to check that this mapping is almost surely functorial.
Following equation \eqref{eq:bayes-comp}, we note that the state with respect to which this functoriality is almost sure is $g\circ f\circ p$, for every state $p$ on $X$.
We therefore need to verify that 
\[\binom X X \xrightarrow{(f, f^\sharp)} \binom Y Y \xrightarrow{(g, g^\sharp)} \binom Z Z \quad \aeq_{g\circ f} \quad \binom X X \xrightarrow{\left(g\circ f, (g\circ f)^\sharp\right)} \binom Z Z \; .\]
This follows from three applications of Definition \ref{def:binv}:
\begin{align*}
& 
\begin{tangle}{(7,5)}[trim y]
	\tgBlank{(0,0)}{white}
	\tgBlank{(1,0)}{white}
	\tgBlank{(2,0)}{white}
	\tgBlank{(3,0)}{white}
	\tgBlank{(4,0)}{white}
	\tgBlank{(5,0)}{white}
	\tgBlank{(6,0)}{white}
	\tgBlank{(0,1)}{white}
	\tgBlank{(1,1)}{white}
	\tgBlank{(2,1)}{white}
	\tgBorderC{(3,1)}{1}{white}{white}
	\tgBorderA{(4,1)}{white}{white}{white}{white}
	\tgBorder{(4,1)}{0}{1}{0}{1}
	\tgBorderA{(5,1)}{white}{white}{white}{white}
	\tgBorder{(5,1)}{0}{1}{0}{1}
	\tgBorderA{(6,1)}{white}{white}{white}{white}
	\tgBorder{(6,1)}{0}{0}{0}{1}
	\tgBorderA{(0,2)}{white}{white}{white}{white}
	\tgBorder{(0,2)}{0}{1}{0}{0}
	\tgBorderA{(1,2)}{white}{white}{white}{white}
	\tgBorder{(1,2)}{0}{1}{0}{1}
	\tgBorderA{(2,2)}{white}{white}{white}{white}
	\tgBorder{(2,2)}{0}{1}{0}{1}
	\tgBorderA{(3,2)}{white}{white}{white}{white}
	\tgBorder{(3,2)}{1}{0}{1}{1}
	\tgBlank{(4,2)}{white}
	\tgBlank{(5,2)}{white}
	\tgBlank{(6,2)}{white}
	\tgBlank{(0,3)}{white}
	\tgBlank{(1,3)}{white}
	\tgBlank{(2,3)}{white}
	\tgBorderC{(3,3)}{0}{white}{white}
	\tgBorderA{(4,3)}{white}{white}{white}{white}
	\tgBorder{(4,3)}{0}{1}{0}{1}
	\tgBorderA{(5,3)}{white}{white}{white}{white}
	\tgBorder{(5,3)}{0}{1}{0}{1}
	\tgBorderA{(6,3)}{white}{white}{white}{white}
	\tgBorder{(6,3)}{0}{0}{0}{1}
	\tgBlank{(0,4)}{white}
	\tgBlank{(1,4)}{white}
	\tgBlank{(2,4)}{white}
	\tgBlank{(3,4)}{white}
	\tgBlank{(4,4)}{white}
	\tgBlank{(5,4)}{white}
	\tgBlank{(6,4)}{white}
	\tgCell[(0,1)]{(0,2)}{p}
	\tgCell[(0,1)]{(1,2)}{f}
	\tgCell[(0,1)]{(2,2)}{g}
	\tgCell[(2,1)]{(4.5,3)}{(g\circ f)^\sharp_p}
\end{tangle} \; = \; 
\begin{tangle}{(5,5)}[trim y]
	\tgBlank{(0,0)}{white}
	\tgBlank{(1,0)}{white}
	\tgBlank{(2,0)}{white}
	\tgBlank{(3,0)}{white}
	\tgBlank{(4,0)}{white}
	\tgBlank{(0,1)}{white}
	\tgBorderC{(1,1)}{1}{white}{white}
	\tgBorderA{(2,1)}{white}{white}{white}{white}
	\tgBorder{(2,1)}{0}{1}{0}{1}
	\tgBorderA{(3,1)}{white}{white}{white}{white}
	\tgBorder{(3,1)}{0}{1}{0}{1}
	\tgBorderA{(4,1)}{white}{white}{white}{white}
	\tgBorder{(4,1)}{0}{0}{0}{1}
	\tgBorderA{(0,2)}{white}{white}{white}{white}
	\tgBorder{(0,2)}{0}{1}{0}{0}
	\tgBorderA{(1,2)}{white}{white}{white}{white}
	\tgBorder{(1,2)}{1}{0}{1}{1}
	\tgBlank{(2,2)}{white}
	\tgBlank{(3,2)}{white}
	\tgBlank{(4,2)}{white}
	\tgBlank{(0,3)}{white}
	\tgBorderC{(1,3)}{0}{white}{white}
	\tgBorderA{(2,3)}{white}{white}{white}{white}
	\tgBorder{(2,3)}{0}{1}{0}{1}
	\tgBorderA{(3,3)}{white}{white}{white}{white}
	\tgBorder{(3,3)}{0}{1}{0}{1}
	\tgBorderA{(4,3)}{white}{white}{white}{white}
	\tgBorder{(4,3)}{0}{0}{0}{1}
	\tgBlank{(0,4)}{white}
	\tgBlank{(1,4)}{white}
	\tgBlank{(2,4)}{white}
	\tgBlank{(3,4)}{white}
	\tgBlank{(4,4)}{white}
	\tgCell[(0,1)]{(0,2)}{p}
	\tgCell[(0,1)]{(2,1)}{f}
	\tgCell[(0,1)]{(3,1)}{g}
\end{tangle} \\
&\qquad = 
\begin{tangle}{(6,5)}[trim y]
	\tgBlank{(3,0)}{white}
	\tgBlank{(4,0)}{white}
	\tgBlank{(0,1)}{white}
	\tgBlank{(1,1)}{white}
	\tgBorderC{(2,1)}{1}{white}{white}
	\tgBorderA{(3,1)}{white}{white}{white}{white}
	\tgBorder{(3,1)}{0}{1}{0}{1}
	\tgBorderA{(4,1)}{white}{white}{white}{white}
	\tgBorder{(4,1)}{0}{1}{0}{1}
	\tgBorderA{(5,1)}{white}{white}{white}{white}
	\tgBorder{(5,1)}{0}{0}{0}{1}
	\tgBorderA{(0,2)}{white}{white}{white}{white}
	\tgBorder{(0,2)}{0}{1}{0}{0}
	\tgBorderA{(1,2)}{white}{white}{white}{white}
	\tgBorder{(1,2)}{0}{1}{0}{1}
	\tgBorderA{(2,2)}{white}{white}{white}{white}
	\tgBorder{(2,2)}{1}{0}{1}{1}
	\tgBlank{(3,2)}{white}
	\tgBlank{(4,2)}{white}
	\tgBlank{(0,3)}{white}
	\tgBlank{(1,3)}{white}
	\tgBorderC{(2,3)}{0}{white}{white}
	\tgBorderA{(3,3)}{white}{white}{white}{white}
	\tgBorder{(3,3)}{0}{1}{0}{1}
	\tgBorderA{(4,3)}{white}{white}{white}{white}
	\tgBorder{(4,3)}{0}{1}{0}{1}
	\tgBorderA{(5,3)}{white}{white}{white}{white}
	\tgBorder{(5,3)}{0}{0}{0}{1}
	\tgBlank{(3,4)}{white}
	\tgBlank{(4,4)}{white}
	\tgCell[(0,1)]{(0,2)}{p}
	\tgCell[(0,1)]{(1,2)}{f}
	\tgCell[(1,1)]{(3.5,3)}{f^\sharp_p}
	\tgCell[(0,1)]{(3.5,1)}{g}
\end{tangle} \; = \; 
\begin{tangle}{(9,5)}[trim y]
	\tgBlank{(0,0)}{white}
	\tgBlank{(1,0)}{white}
	\tgBlank{(2,0)}{white}
	\tgBlank{(3,0)}{white}
	\tgBlank{(4,0)}{white}
	\tgBlank{(5,0)}{white}
	\tgBlank{(6,0)}{white}
	\tgBlank{(7,0)}{white}
	\tgBlank{(8,0)}{white}
	\tgBlank{(0,1)}{white}
	\tgBlank{(1,1)}{white}
	\tgBlank{(2,1)}{white}
	\tgBorderC{(3,1)}{1}{white}{white}
	\tgBorderA{(4,1)}{white}{white}{white}{white}
	\tgBorder{(4,1)}{0}{1}{0}{1}
	\tgBorderA{(5,1)}{white}{white}{white}{white}
	\tgBorder{(5,1)}{0}{1}{0}{1}
	\tgBorderA{(6,1)}{white}{white}{white}{white}
	\tgBorder{(6,1)}{0}{1}{0}{1}
	\tgBorderA{(7,1)}{white}{white}{white}{white}
	\tgBorder{(7,1)}{0}{1}{0}{1}
	\tgBorderA{(8,1)}{white}{white}{white}{white}
	\tgBorder{(8,1)}{0}{0}{0}{1}
	\tgBorderA{(0,2)}{white}{white}{white}{white}
	\tgBorder{(0,2)}{0}{1}{0}{0}
	\tgBorderA{(1,2)}{white}{white}{white}{white}
	\tgBorder{(1,2)}{0}{1}{0}{1}
	\tgBorderA{(2,2)}{white}{white}{white}{white}
	\tgBorder{(2,2)}{0}{1}{0}{1}
	\tgBorderA{(3,2)}{white}{white}{white}{white}
	\tgBorder{(3,2)}{1}{0}{1}{1}
	\tgBlank{(4,2)}{white}
	\tgBlank{(5,2)}{white}
	\tgBlank{(6,2)}{white}
	\tgBlank{(7,2)}{white}
	\tgBlank{(8,2)}{white}
	\tgBlank{(0,3)}{white}
	\tgBlank{(1,3)}{white}
	\tgBlank{(2,3)}{white}
	\tgBorderC{(3,3)}{0}{white}{white}
	\tgBorderA{(4,3)}{white}{white}{white}{white}
	\tgBorder{(4,3)}{0}{1}{0}{1}
	\tgBorderA{(5,3)}{white}{white}{white}{white}
	\tgBorder{(5,3)}{0}{1}{0}{1}
	\tgBorderA{(6,3)}{white}{white}{white}{white}
	\tgBorder{(6,3)}{0}{1}{0}{1}
	\tgBorderA{(7,3)}{white}{white}{white}{white}
	\tgBorder{(7,3)}{0}{1}{0}{1}
	\tgBorderA{(8,3)}{white}{white}{white}{white}
	\tgBorder{(8,3)}{0}{0}{0}{1}
	\tgBlank{(0,4)}{white}
	\tgBlank{(1,4)}{white}
	\tgBlank{(2,4)}{white}
	\tgBlank{(3,4)}{white}
	\tgBlank{(4,4)}{white}
	\tgBlank{(5,4)}{white}
	\tgBlank{(6,4)}{white}
	\tgBlank{(7,4)}{white}
	\tgBlank{(8,4)}{white}
	\tgCell[(0,1)]{(0,2)}{p}
	\tgCell[(0,1)]{(1,2)}{f}
	\tgCell[(0,1)]{(2,2)}{g}
	\tgCell[(1,1)]{(4.5,3)}{g^\sharp_{f\circ p}}
	\tgCell[(1,1)]{(6.5,3)}{f^\sharp_p}
\end{tangle}
\end{align*}
\end{proof}
Note that although Markov categories in general may not admit all Bayesian inverses, we can always restrict to a wide subcategory $\Ca^\dag$ consisting of only those kernels which admit inverses.

\section{Dependent Bayesian Lenses}
Those familiar with categorical semantics of type theory might observe that the construction of Bayesian lenses is similar to that of the \emph{families fibration} \cite[\S1.2]{Jacobs1999Categorical}.
This construction, commonly used in the interpretation of dependent types, constructs a category whose objects and morphisms are set-indexed families of objects and morphisms from another underlying category.
As a fibration this category projects onto the category of sets, by picking out the indexing set or reindexing function under each family.

\begin{definition}[{The families fibration \cite[\S1.2]{Jacobs1999Categorical}}]
The indexed category of families over a category $\Ca$ is the functor $\Fam_\Ca : \Set\op \to \Cat$ defined as so:
\begin{itemize}
\item For a set $X$, $\Fam_\Ca(X)$ is the category whose objects are $X$-indexed families of objects $X \to \Ob(\Ca)$ and whose morphisms  $A \to B$ are families of morphisms $\phi : \{x \in X\} \to \Ca(A_x, B_x)$	.
\item For a function $\alpha : Y \to X$, the induced reindexing functor $\alpha^* : \Fam_\Ca(X) \to \Fam_\Ca(Y)$ is given by pre-composition with alpha:
	\[
		\alpha^*(A) : X \overset{\alpha}\longrightarrow Y
			\overset{A}\longrightarrow \Ob(\Ca)\\
	\]
	\[
		\alpha^*(\phi) : \{y \in Y\} \overset{\alpha}\longrightarrow \{\alpha(x) \in X\}
			\overset{\phi}\longrightarrow \Ca\left(A_{\alpha(y)}, B_{\alpha(y)}\right)
	\]
\end{itemize}
\end{definition}
Alternatively, viewing sets as discrete categories, $\Fam_\Ca$ is the functor that sends $X$ to the functor category $\Cat(X, \Ca)$ and reindexes by precomposition with the function viewed as a functor between discrete categories.

A crucial difference between this and the indexed category in \cref{def:stat-cat} is that the latter is indexed only by sets of the form $\Ca(I, X)$ whereas $\Fam_\Ca(-)$ allows for arbitrary indexing sets.
Having the indexing sets take the form $\Ca(I, -)$ allows us to additionally restrict the reindexing functors to those represented by kernels of $\Ca$.
Formally we can see this as the result of reindexing $\Fam_\Ca(-)$ itself with the state functor:

\begin{definition}[State-indexed families]
The indexed category $\StFam_\Ca : \Ca\op \to \Cat$ is given by the composite of functors:
\[\begin{tikzcd}[ampersand replacement=\&,row sep=scriptsize]
	\Ca\op \& \Set\op \& \Cat \\
	A \& {\Ca(I, A)} \& {\Ca^{\Ca(I, A)}} \\
	B \& {\Ca(I, B)} \& {\Ca^{\Ca(I, B)}}
	\arrow["{\Fam_\Ca}", from=1-2, to=1-3]
	\arrow["{\Ca(I, -)\op}", from=1-1, to=1-2]
	\arrow["f"', from=2-1, to=3-1]
	\arrow["{f \circ -}", from=2-2, to=3-2]
	\arrow["{\alpha \mapsto \alpha(f \circ -)}"', from=3-3, to=2-3]
\end{tikzcd}\]
This has as objects $\Ca(I, X)$-indexed families of objects from $\Ca$ and as morphisms $\Ca(I, X)$-indexed families of kernels from $\Ca$.
\end{definition}

Note that the categories indexed by $\StFam$ are more general than those indexed by $\Stat$:
the objects of $\StFam_\Ca(X)$ are whole functions $\Ca(I, X) \to \Ob(\Ca)$, whereas an object of $\Stat_\Ca(X)$ is merely a single object of $\Ca$.

\begin{proposition}
$\Stat_\Ca$ is a subfunctor of $\StFam_\Ca$ given by restricting to subcategories which consist only of constantly-indexed families of objects i.e. those families $\Ca(I, X) \to \Ob(\Ca)$ which are constant functions.
\qed
\end{proposition}

Observing that $\StFam$ is a generalisation of $\Stat$, we are led to consider whether the fibration constructed from $\StFam$ could profitably be considered as a generalised category of lenses.
By analogy with the use of $\Fam$ in dependent type theory we refer to these as \emph{dependent Bayesian lenses}.

\begin{definition}[Dependent Bayesian Lenses]
We define the category of \emph{dependent Bayesian lenses} over $\Ca$ to be the Grothendieck construction of the fibrewise opposite of $\StFam$,
\[\DBLens(\Ca) = \coprod_{X \in C} \StFam_\Ca(X)\op \; . \]
\end{definition}
Unpacking this definition we have that:
\begin{itemize}
\item Objects of $\DBLens(\Ca)$ are pairs $\binom X A$ where $X \in \Ca$ and $A : \Ca(I, X) \to \Ob(\Ca)$. We think of this as representing a set of dependent pairs, whose elements are 
$\left\langle p, a \right\rangle \in \Ca(I, X) \times A(p)$.
\item Morphisms $\binom X A \to \binom Y B$ are pairs $(f, f^\sharp)$ of a kernel $f : X \to Y$ and a family of kernels 
$f^\sharp : \left\{p \in \Ca(I, X)\right\} \to \Ca(B(f \circ p), A(p))$.
\item The composition is as with the non-dependent version in \cref{def:blens}. Aside from the fact that $f^\sharp$ is here considered as a dependent function, the data of a morphism is the same. It is straightforward to verify that the extra dependent typing data still aligns where appropriate.
\end{itemize}

\begin{remark}
	\label{remark:monoidal-structure}
	The functor $\Ca(I, -)\op : \Ca\op \to \Set\op$ can be made into a lax monoidal functor with structure maps induced by functions $\Ca(I, X \otimes Y) \to \Ca(I, X) \times \Ca(I, Y)$ which map states to pairs of states obtained by deleting either variable.
	Then it follows by results in \cite{moeller2021monoidal} that $\StFam$ is the composition of lax monoidal functors, and so its Grothendieck construction inherits a monoidal structure.
\end{remark}

\section{Support Objects}
Giving an abstract account of Bayes' law, Definition \ref{def:binv} promises to be very important in studying Bayesian statistics synthetically, but it is in some way unsatisfying because it only specifies a morphism up to almost-sure equality.
For example, considering $\mathbf{FinStoch}$, if a distribution $p: I \to X$ is not fully supported, then the inverse of a kernel $X \to Y$ is only specified by the above definition at the points in the support of $p$.
We can work around this ambiguity however by instead considering inverses as kernels between objects representing the supports of distributions.

Fritz \cite{Fritz_2020} proposes a definition for support objects in a Markov category, but does not develop the idea further.
Here we investigate some properties of support objects and explain how they can be used to enhance the theory of abstract Bayesian inversion.

\begin{definition}
Fix a state $p: I \to X$. An object $X_p$ is called a \emph{support of $p$} if $X_p$ represents the covariant functor $(\Ca(X, -) / \aeq_p): \Ca \to \Set$.
\end{definition}

This definition succinctly captures the essential properties of the support of a distribution, but it is quite opaque and does not encourage intuition.
If an object is a support of another object, intuitively it should behave as a subspace, having an inclusion map which satisfies some extra properties. 
We state this formally in terms of the existence of certain section-retraction kernels representing this inclusion.

\begin{restatable}{proposition}{supportObjectsThm}
\label{thm:support-objects}
$X_p$ is a support of $p: I \to X$ if and only if there is a section-retraction pair $X_p \xrightarrow{i} X \xrightarrow{r} X_p$ such that for any kernels $f, g: X \to Y$ we have $f \aeq_p g \iff f\circ i = g \circ i$.
\end{restatable}

\begin{proof}
Assume $X_p$ is a support. This means we have a natural isomorphism 
$\Phi: (\Ca(X, -)/\aeq_p) \to \Ca(X_p, -) $.
We take $i = \Phi_X(\id_X)$, then we see from the following naturality square that the action of $\Phi$ must be to pre-compose representative morphisms with $i$:
\[\begin{tikzcd}[ampersand replacement=\&]
	\textcolor{rgb,255:red,128;green,128;blue,128}{i} \&\&\& \textcolor{rgb,255:red,130;green,130;blue,130}{[\id_X]_{\aeq_p}} \\
	\& {\Ca(X_p, X)} \& {\Ca(X, X)/\aeq_p} \\
	\& {\Ca(X_p, Y)} \& {\Ca(X, Y)/\aeq_p} \\
	\textcolor{rgb,255:red,128;green,128;blue,128}{f \circ i} \&\&\& \textcolor{rgb,255:red,128;green,128;blue,128}{[f]_{\aeq_p}}
	\arrow["{\Ca(X, f)/\aeq_p}", from=2-3, to=3-3]
	\arrow["{\Ca(X_p, f)}"', from=2-2, to=3-2]
	\arrow["{\Phi_X}"', from=2-3, to=2-2]
	\arrow["{\Phi_Y}", from=3-3, to=3-2]
	\arrow[draw={rgb,255:red,128;green,128;blue,128}, shorten <=11pt, shorten >=11pt, maps to, from=1-4, to=4-4]
	\arrow[draw={rgb,255:red,128;green,128;blue,128}, shorten <=30pt, shorten >=30pt, maps to, from=4-4, to=4-1]
	\arrow[draw={rgb,255:red,128;green,128;blue,128}, shorten <=11pt, shorten >=11pt, maps to, from=1-1, to=4-1]
	\arrow[draw={rgb,255:red,128;green,128;blue,128}, shorten <=14pt, shorten >=29pt, maps to, from=1-4, to=1-1]
\end{tikzcd}\]

Hence we establish the property that pre-composition by $i$ is an isomorphism between kernels from $X$ and $\aeq_p$-equivalences classes of kernels from $X_p$.
We further have that $\id_{X_p} = \Phi(\Phi\inv(\id_{X_p})) = \Phi\inv(\id_{X_p})\circ i$, so we can take the retract to be $r = \Phi\inv(\id)$.

Conversely, given such an $i$ and $r$, it is clear that pre-composition by $i$ defines a function $\Ca(X, Y) \to \Ca(X_p, Y)$ natural in $Y$ and the assumed property of $i$ guarantees that this is a bijection from $\aeq_p$-equivalence classes.
Finally we have that $i \circ r \circ i = i$, so $i \circ r \aeq_p \id_{X_p}$.
Hence pre-composition by $r$ is an inverse to $(-) \circ i: (\Ca(X, -)/\aeq_p) \to \Ca(X_p, -)$.
\end{proof}

While support objects are not necessarily unique, they must be unique up to isomorphism, since two support objects for the same distribution must by definition represent the same presheaf.
When we disuss a given support object we therefore really mean a given support object along with a choice of section and retraction.

Now, if we have a distribution on $X$, $p: I \to X$, and a kernel $f: X \to Y$ we can push $p$ forward to a distribution $f \circ p$ on $Y$.
So $f$ restricts to a kernel $\restr f p = r\circ f \circ i: X_{p} \to Y_{f \circ p}$.
Dually we have an inclusion of kernels $g : X_p \to Y_q$ into $\incl g = i \circ g \circ r: X \to Y$.
Using these there is an obvious adjustment to the definition of Bayesian inversion in order to capture Bayesian inverses between supports:

\begin{definition}[Bayesian-inverse-with-support]
Fix kernels $p: I \to X$ and $f: X \to Y$ with support objects $X_p$ and $Y_{f \circ p}$.
We call a kernel $f^\sharp_p: Y_{f \circ p} \to X_p$  a \emph{Bayesian inverse with support} if $\incl{f^\sharp_p}$ is an ordinary Bayesian inverse for $f$ at $p$.
\end{definition}

This is very similar to the previous definition.
However the presence of support objects in $f^\sharp_p$'s domain and codomain mean that the kernel is more canonical than an ordinary Bayesian inverse.
In fact for any ordinary inverse $f^\sharp_p$, its restriction to the distribution $f \circ p$ is an inverse-with-support. 

\begin{restatable}[Unique Bayesian Inversion]{theorem}{uniqueInversionThm}
\label{thm:unique-inversion}
Fix kernels $f: X \to Y$ and $p: I \to X$, and support objects $X_p$ and $Y_{f \circ p}$.
If $f$ has an ordinary Bayesian inverse at $p$, then there is a unique inverse-with-support to $f$ at $p$.
\end{restatable}
\begin{proof}
See \cref{appendix:support-calculus}.	
\end{proof}

This final result suggests that we can now define a functor that picks out the canonical inverse for a given kernel.
Recall that in \cref{thm:binv-functorial} we were almost able to show that Bayesian inversion defines a functor $\Ca \to \BLens(\Ca)$ except that the functoriality constraint only holds up to almost-equality. 
Using support objects we improve this into strict functoriality between kernels. Bayesian inversion with supports departs from the previous situation slightly in that the domain and codomain of the inverse kernels vary as the indexing distribution varies. Fortunately this is exactly the extra level of generality afforded us by dependent Bayesian lenses.

\begin{proposition}
\label{thm:exact-inversion}
If $\Ca$ has Bayesian inverses for every kernel, and support objects for every distribution $I \to X$ then the fibred category $\DBLens(\Ca)$ has a section $T : \Ca \to \DBLens(\Ca)$ sending kernels to families of their Bayesian inverses between support objects.
\end{proposition}
\begin{proof}
We write explicitly the mapping defining $T$.
On objects $T(X) = \binom X {S_X}$ where $S_X : \Ca(I, X) \to \Ob(\Ca)$ maps $p$ to a choice of support object $X_p$.
On kernels $f : X \to Y$, we have $T(X) = (f, f^\sharp) : \binom X {S_X} \to \binom Y {S_Y}$ where for each $p \in \Ca(I, X)$, $f^\sharp_p : Y_{f \circ p} \to X_p$ is given by the Bayesian inverse with support of $f$ at $p$.

The functoriality of $T$ follows similarly to in \cref{thm:binv-functorial}.
Briefly, if $f$ and $g$ are composable Bayesian inverses with support, then $\incl f$ and $\incl g$ are composable ordinary inverses, so by \cref{thm:binv-functorial} their composition must be an inverse. By the definition of the inclusions it is immediately that $\incl f \circ \incl g = \incl {f \circ g}$, so $f \circ g$ is an inverse-with-support. Then by uniqueness, the fact that the composition $f \circ g$ is a Bayesian inverse implies the strict equality required for functoriality.
\end{proof}

\begin{remark}
\cref{thm:exact-inversion} proves the compatibility of sequential composition in a Markov category with Bayesian inversion, however there is another axis for composition in our graphical notation. 
This is the parallel composition, or monoidal product of $\Ca$.
To ask for this structure to be compatible with inversion is to ask for the functor $T$ to respect the monoidal structure.
Indeed as we noted in \cref{remark:monoidal-structure}, we may naturally define a monoidal product on $\DBLens(\Ca)$, then $T$ straightforwardly inherits the structure of a \emph{lax monoidal functor} with respect to this.
\end{remark}

\begin{remark}
Most of the results developed in this section relied on the existence of support objects. Indeed this is satisfied for our first two examples \cref{ex:finstoch} and \cref{ex:gauss}, but this is quite a strong assumption in general.
The approach to functorial Bayesian inversion taken in \cite{Cho_Jacobs_2019} and later expanded on in \cite{Fritz_2020} is to work in a category $\ProbStoch(\Ca)$ whose objects are equipped with a choice of distribution and whose morphisms are quotiented by almost-sure equality with respect to these chosen distributions. This makes exact functoriality straightforward to prove, but it detracts from the interpretation of morphisms as stochastic functions. Working in $\ProbStoch(\Ca)$ we are unable to think of Bayesian updating as a dynamic process which maps observations to new distributions, since the resultant distribution is already encoded in the codomain.

We could instead take a    hybrid approach, using dependent Bayesian lenses where the forwards kernel is still valued in $\Ca$ but the `backwards’ mapping is in $\ProbStoch(\Ca)$. This means the type of a general inverse can become a function $\{p \in \Ca(I, X)\} \to \ProbStoch(\Ca)((Y, f\circ p), (X, p))$. This provides a lot of the same theoretical niceties of support objects, including exact functoriality, while still applying in a more general context. We leave the further investigation of this construction to future work.
\end{remark}

\section{Example: Estimating Transition Probabilities in Markov Chains}
To illustrate the compositional nature of Bayesian inversion in practice we consider how to view a typical Bayesian inference problem in this framework.
Namely we use as our example, the common problem of learning transition probabilities for a Markov chain from an observed sequence of states \cite{Wu2010}.

In the category $\FinStoch$, a Markov chain is nothing but an endomorphism $t : S \to S$ for some finite set $S$ of \emph{states}, with an initial state distribution $s : I \to S$. Often we do not know the actual transition probabilities, but instead have a family of distributions dependent on another external parameter, $t : S \otimes \Theta \to S$. In such a situation we want to choose an optimal value of $\Theta$ based on a sequence of observations of states in the past.
A common technique for choosing this value is to consider the observed state sequence as an element of the space $S \otimes \ldots \otimes S$, with which we can perform a Bayesian update. That is, given a state sequence $(s_1, \ldots, s_n)$ we would like to compute a posterior distribution $\pr(\Theta | s_1, \ldots, s_n)$. 

From the data involved we can define a family of maps $f_n : \Theta \to S^{\otimes n}$ describing the distribution over $n$-length state sequences corresponding to any value of $\Theta$ (see \cref{fig:composite-kernel} for an example). Computing the Bayesian inverse of $f_n$ we obtain kernels which map sequences $S^{\otimes n}$ to posterior distributions over $\Theta$. Of course, as we have already discussed, in order to compute this all we need is the Bayesian inverse for the transition matrix $t$. Knowing the functorial relationship between `forward' models $f_n$ and their corresponding inverse models, we can build an inverse for $f_n$ out of many copies of the inverse for $t$.

\begin{figure}[h]
\caption{An example of a `state-trace' kernel obtained from a Markov chain, yielding state sequences of length 4.}
\label{fig:composite-kernel}
\centering
$f_4$ = 
\begin{tangle}{(14,11)}[trim y]
	\tgBlank{(0,0)}{white}
	\tgBlank{(1,0)}{white}
	\tgBlank{(2,0)}{white}
	\tgBlank{(3,0)}{white}
	\tgBlank{(4,0)}{white}
	\tgBlank{(5,0)}{white}
	\tgBlank{(6,0)}{white}
	\tgBlank{(7,0)}{white}
	\tgBlank{(8,0)}{white}
	\tgBlank{(9,0)}{white}
	\tgBlank{(10,0)}{white}
	\tgBlank{(11,0)}{white}
	\tgBlank{(12,0)}{white}
	\tgBlank{(13,0)}{white}
	\tgBorderA{(0,1)}{white}{white}{white}{white}
	\tgBorder{(0,1)}{0}{1}{0}{1}
	\tgBorderA{(1,1)}{white}{white}{white}{white}
	\tgBorder{(1,1)}{0}{1}{1}{1}
	\tgBorderA{(2,1)}{white}{white}{white}{white}
	\tgBorder{(2,1)}{0}{1}{0}{1}
	\tgBorderA{(3,1)}{white}{white}{white}{white}
	\tgBorder{(3,1)}{0}{1}{0}{1}
	\tgBorderA{(4,1)}{white}{white}{white}{white}
	\tgBorder{(4,1)}{0}{1}{0}{1}
	\tgBorderA{(5,1)}{white}{white}{white}{white}
	\tgBorder{(5,1)}{0}{1}{1}{1}
	\tgBorderA{(6,1)}{white}{white}{white}{white}
	\tgBorder{(6,1)}{0}{1}{0}{1}
	\tgBorderA{(7,1)}{white}{white}{white}{white}
	\tgBorder{(7,1)}{0}{1}{0}{1}
	\tgBorderA{(8,1)}{white}{white}{white}{white}
	\tgBorder{(8,1)}{0}{1}{0}{1}
	\tgBorderA{(9,1)}{white}{white}{white}{white}
	\tgBorder{(9,1)}{0}{1}{1}{1}
	\tgBorderA{(10,1)}{white}{white}{white}{white}
	\tgBorder{(10,1)}{0}{1}{0}{1}
	\tgBorderA{(11,1)}{white}{white}{white}{white}
	\tgBorder{(11,1)}{0}{1}{0}{1}
	\tgBorderA{(12,1)}{white}{white}{white}{white}
	\tgBorder{(12,1)}{0}{1}{0}{1}
	\tgBorderA{(13,1)}{white}{white}{white}{white}
	\tgBorder{(13,1)}{0}{0}{0}{1}
	\tgBlank{(0,2)}{white}
	\tgBorderA{(1,2)}{white}{white}{white}{white}
	\tgBorder{(1,2)}{1}{0}{1}{0}
	\tgBlank{(2,2)}{white}
	\tgBlank{(3,2)}{white}
	\tgBlank{(4,2)}{white}
	\tgBorderA{(5,2)}{white}{white}{white}{white}
	\tgBorder{(5,2)}{1}{0}{1}{0}
	\tgBlank{(6,2)}{white}
	\tgBlank{(7,2)}{white}
	\tgBlank{(8,2)}{white}
	\tgBorderC{(9,2)}{0}{white}{white}
	\tgBorderA{(10,2)}{white}{white}{white}{white}
	\tgBorder{(10,2)}{0}{0}{1}{1}
	\tgBlank{(11,2)}{white}
	\tgBlank{(12,2)}{white}
	\tgBlank{(13,2)}{white}
	\tgBlank{(0,3)}{white}
	\tgBorderA{(1,3)}{white}{white}{white}{white}
	\tgBorder{(1,3)}{1}{0}{1}{0}
	\tgBlank{(2,3)}{white}
	\tgBlank{(3,3)}{white}
	\tgBlank{(4,3)}{white}
	\tgBorderA{(5,3)}{white}{white}{white}{white}
	\tgBorder{(5,3)}{1}{0}{1}{0}
	\tgBlank{(6,3)}{white}
	\tgBlank{(7,3)}{white}
	\tgBlank{(8,3)}{white}
	\tgBlank{(9,3)}{white}
	\tgBorderA{(10,3)}{white}{white}{white}{white}
	\tgBorder{(10,3)}{1}{1}{1}{0}
	\tgBorderA{(11,3)}{white}{white}{white}{white}
	\tgBorder{(11,3)}{0}{1}{0}{1}
	\tgBorderA{(12,3)}{white}{white}{white}{white}
	\tgBorder{(12,3)}{0}{1}{0}{1}
	\tgBorderA{(13,3)}{white}{white}{white}{white}
	\tgBorder{(13,3)}{0}{1}{0}{1}
	\tgBlank{(0,4)}{white}
	\tgBorderA{(1,4)}{white}{white}{white}{white}
	\tgBorder{(1,4)}{1}{0}{1}{0}
	\tgBlank{(2,4)}{white}
	\tgBlank{(3,4)}{white}
	\tgBlank{(4,4)}{white}
	\tgBorderC{(5,4)}{0}{white}{white}
	\tgBorderA{(6,4)}{white}{white}{white}{white}
	\tgBorder{(6,4)}{0}{0}{1}{1}
	\tgBlank{(7,4)}{white}
	\tgBlank{(8,4)}{white}
	\tgBorderC{(9,4)}{1}{white}{white}
	\tgBorderA{(10,4)}{white}{white}{white}{white}
	\tgBorder{(10,4)}{1}{0}{0}{1}
	\tgBlank{(11,4)}{white}
	\tgBlank{(12,4)}{white}
	\tgBlank{(13,4)}{white}
	\tgBlank{(0,5)}{white}
	\tgBorderA{(1,5)}{white}{white}{white}{white}
	\tgBorder{(1,5)}{1}{0}{1}{0}
	\tgBlank{(2,5)}{white}
	\tgBlank{(3,5)}{white}
	\tgBlank{(4,5)}{white}
	\tgBlank{(5,5)}{white}
	\tgBorderA{(6,5)}{white}{white}{white}{white}
	\tgBorder{(6,5)}{1}{1}{1}{0}
	\tgBorderA{(7,5)}{white}{white}{white}{white}
	\tgBorder{(7,5)}{0}{1}{0}{1}
	\tgBorderA{(8,5)}{white}{white}{white}{white}
	\tgBorder{(8,5)}{0}{1}{0}{1}
	\tgBorderA{(9,5)}{white}{white}{white}{white}
	\tgBorder{(9,5)}{1}{1}{0}{1}
	\tgBorderA{(10,5)}{white}{white}{white}{white}
	\tgBorder{(10,5)}{0}{1}{0}{1}
	\tgBorderA{(11,5)}{white}{white}{white}{white}
	\tgBorder{(11,5)}{0}{1}{0}{1}
	\tgBorderA{(12,5)}{white}{white}{white}{white}
	\tgBorder{(12,5)}{0}{1}{0}{1}
	\tgBorderA{(13,5)}{white}{white}{white}{white}
	\tgBorder{(13,5)}{0}{1}{0}{1}
	\tgBlank{(0,6)}{white}
	\tgBorderC{(1,6)}{0}{white}{white}
	\tgBorderA{(2,6)}{white}{white}{white}{white}
	\tgBorder{(2,6)}{0}{0}{1}{1}
	\tgBlank{(3,6)}{white}
	\tgBlank{(4,6)}{white}
	\tgBorderC{(5,6)}{1}{white}{white}
	\tgBorderA{(6,6)}{white}{white}{white}{white}
	\tgBorder{(6,6)}{1}{0}{0}{1}
	\tgBlank{(7,6)}{white}
	\tgBlank{(8,6)}{white}
	\tgBlank{(9,6)}{white}
	\tgBlank{(10,6)}{white}
	\tgBlank{(11,6)}{white}
	\tgBlank{(12,6)}{white}
	\tgBlank{(13,6)}{white}
	\tgBlank{(0,7)}{white}
	\tgBlank{(1,7)}{white}
	\tgBorderA{(2,7)}{white}{white}{white}{white}
	\tgBorder{(2,7)}{1}{1}{1}{0}
	\tgBorderA{(3,7)}{white}{white}{white}{white}
	\tgBorder{(3,7)}{0}{1}{0}{1}
	\tgBorderA{(4,7)}{white}{white}{white}{white}
	\tgBorder{(4,7)}{0}{1}{0}{1}
	\tgBorderA{(5,7)}{white}{white}{white}{white}
	\tgBorder{(5,7)}{1}{1}{0}{1}
	\tgBorderA{(6,7)}{white}{white}{white}{white}
	\tgBorder{(6,7)}{0}{1}{0}{1}
	\tgBorderA{(7,7)}{white}{white}{white}{white}
	\tgBorder{(7,7)}{0}{1}{0}{1}
	\tgBorderA{(8,7)}{white}{white}{white}{white}
	\tgBorder{(8,7)}{0}{1}{0}{1}
	\tgBorderA{(9,7)}{white}{white}{white}{white}
	\tgBorder{(9,7)}{0}{1}{0}{1}
	\tgBorderA{(10,7)}{white}{white}{white}{white}
	\tgBorder{(10,7)}{0}{1}{0}{1}
	\tgBorderA{(11,7)}{white}{white}{white}{white}
	\tgBorder{(11,7)}{0}{1}{0}{1}
	\tgBorderA{(12,7)}{white}{white}{white}{white}
	\tgBorder{(12,7)}{0}{1}{0}{1}
	\tgBorderA{(13,7)}{white}{white}{white}{white}
	\tgBorder{(13,7)}{0}{1}{0}{1}
	\tgBlank{(0,8)}{white}
	\tgBorderC{(1,8)}{1}{white}{white}
	\tgBorderA{(2,8)}{white}{white}{white}{white}
	\tgBorder{(2,8)}{1}{0}{0}{1}
	\tgBlank{(3,8)}{white}
	\tgBlank{(4,8)}{white}
	\tgBlank{(5,8)}{white}
	\tgBlank{(6,8)}{white}
	\tgBlank{(7,8)}{white}
	\tgBlank{(8,8)}{white}
	\tgBlank{(9,8)}{white}
	\tgBlank{(10,8)}{white}
	\tgBlank{(11,8)}{white}
	\tgBlank{(12,8)}{white}
	\tgBlank{(13,8)}{white}
	\tgBorderA{(0,9)}{white}{white}{white}{white}
	\tgBorder{(0,9)}{0}{1}{0}{0}
	\tgBorderA{(1,9)}{white}{white}{white}{white}
	\tgBorder{(1,9)}{1}{1}{0}{1}
	\tgBorderA{(2,9)}{white}{white}{white}{white}
	\tgBorder{(2,9)}{0}{1}{0}{1}
	\tgBorderA{(3,9)}{white}{white}{white}{white}
	\tgBorder{(3,9)}{0}{1}{0}{1}
	\tgBorderA{(4,9)}{white}{white}{white}{white}
	\tgBorder{(4,9)}{0}{1}{0}{1}
	\tgBorderA{(5,9)}{white}{white}{white}{white}
	\tgBorder{(5,9)}{0}{1}{0}{1}
	\tgBorderA{(6,9)}{white}{white}{white}{white}
	\tgBorder{(6,9)}{0}{1}{0}{1}
	\tgBorderA{(7,9)}{white}{white}{white}{white}
	\tgBorder{(7,9)}{0}{1}{0}{1}
	\tgBorderA{(8,9)}{white}{white}{white}{white}
	\tgBorder{(8,9)}{0}{1}{0}{1}
	\tgBorderA{(9,9)}{white}{white}{white}{white}
	\tgBorder{(9,9)}{0}{1}{0}{1}
	\tgBorderA{(10,9)}{white}{white}{white}{white}
	\tgBorder{(10,9)}{0}{1}{0}{1}
	\tgBorderA{(11,9)}{white}{white}{white}{white}
	\tgBorder{(11,9)}{0}{1}{0}{1}
	\tgBorderA{(12,9)}{white}{white}{white}{white}
	\tgBorder{(12,9)}{0}{1}{0}{1}
	\tgBorderA{(13,9)}{white}{white}{white}{white}
	\tgBorder{(13,9)}{0}{1}{0}{1}
	\tgBlank{(0,10)}{white}
	\tgBlank{(1,10)}{white}
	\tgBlank{(2,10)}{white}
	\tgBlank{(3,10)}{white}
	\tgBlank{(4,10)}{white}
	\tgBlank{(5,10)}{white}
	\tgBlank{(6,10)}{white}
	\tgBlank{(7,10)}{white}
	\tgBlank{(8,10)}{white}
	\tgBlank{(9,10)}{white}
	\tgBlank{(10,10)}{white}
	\tgBlank{(11,10)}{white}
	\tgBlank{(12,10)}{white}
	\tgBlank{(13,10)}{white}
	\tgCell[(1,2)]{(2.5,7)}{t}
	\tgCell[(1,2)]{(6.5,5)}{t}
	\tgCell[(1,2)]{(10.5,3)}{t}
	\tgDot{(13,1)}{}
	\tgCell[(0,1)]{(0,9)}{s}
	\tgAxisLabel{(0,1.5)}{east}{\Theta}
	\tgAxisLabel{(14,3.5)}{west}{S}
	\tgAxisLabel{(14,5.5)}{west}{S}
	\tgAxisLabel{(14,7.5)}{west}{S}
	\tgAxisLabel{(14,9.5)}{west}{S}
\end{tangle}
\end{figure}

\cref{fig:composite-kernel} depicts a typical composite which we may want to invert. It consists only of four kinds of non-trivial cell: the transition matrix $t$, the initial state distribution $s$, copy maps $\dup_S : S \to S \otimes S$ and $\dup_\Theta : \Theta \to \Theta \otimes \Theta$, and a delete map $\delete_\Theta : \Theta \to I$.
As $t$ is an arbitrary kernel there is little we can say about its Bayesian inverse, but we can fully characterise the inverses for the other kinds of cells.
The inverse of $s$ at $p$ should have a type $s^\sharp_p : S_p \to I$, but by the uniqueness of the delete map this is equal to $\delete_{S_p} : S_p \to I$.
Dually, filling in the definition of Bayesian inversion for $\delete_\Theta$ at $p$ requires that $(\delete_\Theta)^\sharp_p$ is exactly $p : I \to \Theta$.

Finally we consider the Bayesian inverse of $\dup_S$. Of course the case for $\dup_\Theta$ is identical.
In the case of non-dependent Bayesian lenses, without supports, it is difficult to understand what these copy maps do.
The copy operation is supported on only a small subset of its domain. Intuitively this is the diagonal set $\{(s, s) | s \in S \}$, but in the abstract case this set notation is not valid. We may wonder if there is a way that we can describe this using only the axiomatic setting we have been working with so far.
Indeed, we can show that there is in fact an isomorphism of the support object $(S \otimes S)_{\dup_\Theta \circ p} \cong S_p$ and moreover, the morphism witnessing this isomorphism is exactly the Bayesian inverse of $\dup$. 

This isomorphism is a very powerful fact. It enforces in the type of a morphism, using the notion of dependent typing present in $\DBLens(\Ca)$, that the observations we make from a process involving copying should preserve the exact equalities expected given our knowledge about the generative processes from which the observations originated.

It is not difficult to derive an explicit formula for this with traditional probability theory, but our structural viewpoint can offer a more pedagogical approach to describing inference algorithms.
Using a 2-dimensional syntax for kernels and understanding the compositional nature of Bayesian inversion, a lot of the apparent complexity of equations in statistics is clarified by applying operations piecewise on string diagrams.
It additionally becomes straightforward to see how to extend this to more complicated scenarios. For example, a hidden Markov model is obtained from the above example simply by postcomposing each wire with an additional kernel $o: S \to O$. Then compositionality described exactly how to extend a solution to account for the additional mapping.

This approach provides not just additional pedagogy, but also is very amenable to algorithmic study: as future work, we hope to investigate ways in which, by isolating the basic repeated units and their compositional structure, we can automatically generate a lot of the additional data required and can perhaps facilitate certain optimisations.

\section{Further work}
Our work situates Bayesian inversion amongst a range of bidirectional processes observed in different contexts, each of which exhibit lens (or lens-like) structure and `cybernetic' application \cite{Capucci2021Towards}:
reverse-mode automatic differentiation (generalizing backpropagation of error) \cite{gavranovic_etal_compositional_gradient}; economic games \cite{Ghani_Hedges_Winschel_Zahn_2018}; reinforcement learning \cite{Hedges2022Value}; and database updating \cite{Foster_Greenwald_Moore_Pierce_Schmitt_2007}.
It was in this latter setting, in the context of functional programming, that lenses were first described, and generalizations of lenses remain in popular use in that setting.
This points to a first avenue for future work: a new approach to probabilistic programming that incorporates both Bayesian and differential updating, extending the currently popular use of deterministic lenses, and with general application to cybernetic systems.

Probabilistic programming languages allow the programmer essentially to construct Markov kernels in a compositional way using a standard programming language syntax, and perform inference on the resulting models \cite{vandeMeent2021Introduction}.
Typically, performing inference is not transparently compositional, and so we hope that our results will lead to improvements here, for instance by allowing the programmer to `amortize' parts of an inference process.
Perhaps more importantly, our approach could lead to the use of a dependent type system to statically rule out events of measure zero.
Furthermore, in a generically bidirectional language of the kind we envisage, we expect that compilers will be able to share optimizations across differential and probabilistic parts.

We expect these ideas not only to be of use to language designers therefore, but also to users.
In many applications of probabilistic inference, one first constructs a ``joint model'': a joint state across all the variables of interest, which may factorise according to a Bayesian network or other graphical model.
Because computing exact Bayesian inversions involves marginalization (normalization: the sum or integral in the Chapman-Kolmogorov equation), and such a computation is often intractable, one usually resorts to approximate methods, and this is where the `extra' morphisms in categories of Bayesian lenses come in: they can be understood as approxmate inverses.

By parameterizing these lenses, one can associate to them \textit{loss functions} that characterize ``how far'' a given inversion is from optimality\footnote{
  Examples of such loss functions include the relative entropy (Kullback-Leibler divergence) and variational free energy (``evidence lower bound'').};
like the inversions themselves, these losses depend on both the priors and the observations, and they again compose according to a lens pattern.
The third author, in the unpublished works \cite{Smithe_2021,Smithe2022Mathematical}, has begun developing this programme, based on ideas from compositional game theory \cite{Ghani_Hedges_Winschel_Zahn_2018}.
It results in an account of approximate inference whose algorithms are ``correct by construction'', which may be an advantage over traditional methods which simply start from a given joint distribution.
Moreover, because the loss functions are `local' to each lens, the resulting framework captures the compositional structure seemingly exhibited by ``predictive coding'' neural circuits in the brain \cite{Bastos2012Canonical}.
Similarly, by making use of the lens structure presented here, the framework suggests a formal unification of backprop and predictive coding, as sought by various authors \cite{Millidge2022Predictive,Rosenbaum2022Relationship}, and it reveals connections to the method of backward induction in reinforcement learning \cite{Hedges2022Value}.
We hope that future developments make use of these relationships, so that we may build intelligent systems that are both efficient and well-understood.

%
%
%
\bibliography{MFCS23.bib}

\appendix


\section{Uniqueness of Bayesian Inversion}
\label{appendix:support-calculus}
\begin{notation}
In the graphical language we denote the inclusion and restriction kernels for a support object respectively as follows:
\begin{tangle}{(4,1)}[trim y]
	\tgBorderA{(0,0)}{white}{white}{white}{white}
	\tgBorder{(0,0)}{0}{1}{0}{1}
	\tgBorderA{(3,0)}{white}{white}{white}{white}
	\tgBorder{(3,0)}{0}{1}{0}{1}
	\tgRestr{(3,0)}{}
	\tgIncl{(0,0)}{}
	\tgAxisLabel{(0,0.5)}{east}{X_p}
	\tgAxisLabel{(1,0.5)}{west}{X}
	\tgAxisLabel{(3,0.5)}{east}{X}
	\tgAxisLabel{(4,0.5)}{west}{X_p}
\end{tangle}.
\end{notation}

This is a convenient shorthand for indicating the section-retract relation between sections and restrictions, however care must be taken when calculating with these. 
Since we are not labelling the triangles, we should only pair inclusions and restrictions corresponding to the same states; the characterising equations only hold for such pairs.
However in practice it is difficult to naturally arrive at a situation where this caveat can be an issue, so this is less of a shortcoming that it might seem.

We have strict equality of inclusions followed by restrictions, but in the converse we have almost-equality of restrictions followed by inclusions.

\begin{lemma}
\label{thm:restr-incl-almost-inv}
Restriction is almost inverse to inclusion:
\[X \xrightarrow{r} X_p \xrightarrow{i} X \quad \aeq_p \quad X \xrightarrow{\id} X.\]
\end{lemma}

\begin{proof}
Note that $i \circ r \circ i = i$ so by the quotienting property of $(-) \circ i$ we have $i \circ r \aeq_p \id$.
\end{proof}

\begin{proposition}
\label{bayesian_inverse_bijection}
Fix kernels $f: X \to Y$ and $p: I \to X$, and support objects $X_p$ and $Y_{f p}$.
Then inverses-with-support of $f$ at $p$ are in bijection with $\aeq_{f p}$-equivalence classes of ordinary Bayesian inverses.
\end{proposition}
\begin{proof}

We first exhibit a map $\Psi$ from inverses-with-support to ordinary inverses.
Let $g: Y_{f p} \to X_p$ be a Bayesian inverse with support of $f$ at $p$.
By the definition of inverses with support this means that $\Psi(g) \defeq \incl g$ is an ordinary Bayesian inverse.

Conversely if $h: Y \to X$ is an ordinary Bayesian inverse to $f$ at $p$, we want to define an inverse with support
$\tilde{\Psi}(h) : Y_{fp} \to X_p$.
As a first guess we might try the restriction $\restr h {fp}$ but we see that the codomain of this kernel does not match.
Namely we have $\restr h {fp} : Y_{fp} \to X_{hfp}$ but require a codomain of $X_p$.
In fact we can verify that $hfp = p$:
since $h$ is a Bayesian inverse of $f$ at $p$ we have that
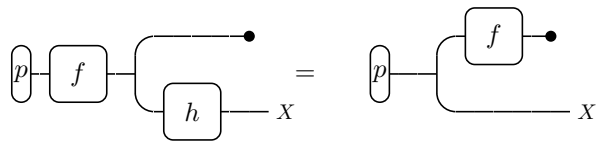
\begin{equation*}
\begin{tangle}{(7,5)}
	\tgBlank{(0,0)}{white}
	\tgBlank{(1,0)}{white}
	\tgBlank{(2,0)}{white}
	\tgBlank{(3,0)}{white}
	\tgBlank{(4,0)}{white}
	\tgBlank{(5,0)}{white}
	\tgBlank{(6,0)}{white}
	\tgBlank{(0,1)}{white}
	\tgBlank{(1,1)}{white}
	\tgBlank{(2,1)}{white}
	\tgBorderC{(3,1)}{1}{white}{white}
	\tgBorderA{(4,1)}{white}{white}{white}{white}
	\tgBorder{(4,1)}{0}{1}{0}{1}
	\tgBorderA{(5,1)}{white}{white}{white}{white}
	\tgBorder{(5,1)}{0}{1}{0}{1}
	\tgBorderA{(6,1)}{white}{white}{white}{white}
	\tgBorder{(6,1)}{0}{1}{0}{1}
	\tgBorderA{(0,2)}{white}{white}{white}{white}
	\tgBorder{(0,2)}{0}{1}{0}{0}
	\tgBorderA{(1,2)}{white}{white}{white}{white}
	\tgBorder{(1,2)}{0}{1}{0}{1}
	\tgBorderA{(2,2)}{white}{white}{white}{white}
	\tgBorder{(2,2)}{0}{1}{0}{1}
	\tgBorderA{(3,2)}{white}{white}{white}{white}
	\tgBorder{(3,2)}{1}{0}{1}{1}
	\tgBlank{(4,2)}{white}
	\tgBlank{(5,2)}{white}
	\tgBlank{(6,2)}{white}
	\tgBlank{(0,3)}{white}
	\tgBlank{(1,3)}{white}
	\tgBlank{(2,3)}{white}
	\tgBorderC{(3,3)}{0}{white}{white}
	\tgBorderA{(4,3)}{white}{white}{white}{white}
	\tgBorder{(4,3)}{0}{1}{0}{1}
	\tgBorderA{(5,3)}{white}{white}{white}{white}
	\tgBorder{(5,3)}{0}{1}{0}{1}
	\tgBorderA{(6,3)}{white}{white}{white}{white}
	\tgBorder{(6,3)}{0}{0}{0}{1}
	\tgBlank{(0,4)}{white}
	\tgBlank{(1,4)}{white}
	\tgBlank{(2,4)}{white}
	\tgBlank{(3,4)}{white}
	\tgBlank{(4,4)}{white}
	\tgBlank{(5,4)}{white}
	\tgBlank{(6,4)}{white}
	\tgCell[(0,1)]{(0,2)}{p}
	\tgCell[(1,1)]{(1.5,2)}{f}
	\tgCell[(1,1)]{(4.5,1)}{h}
	\tgDot{(6,3)}{}
	\tgAxisLabel{(7,1.5)}{west}{X}
\end{tangle}
=
\begin{tangle}{(6,5)}
	\tgBlank{(0,0)}{white}
	\tgBlank{(1,0)}{white}
	\tgBlank{(2,0)}{white}
	\tgBlank{(3,0)}{white}
	\tgBlank{(4,0)}{white}
	\tgBlank{(5,0)}{white}
	\tgBlank{(0,1)}{white}
	\tgBlank{(1,1)}{white}
	\tgBorderC{(2,1)}{1}{white}{white}
	\tgBorderA{(3,1)}{white}{white}{white}{white}
	\tgBorder{(3,1)}{0}{1}{0}{1}
	\tgBorderA{(4,1)}{white}{white}{white}{white}
	\tgBorder{(4,1)}{0}{1}{0}{1}
	\tgBorderA{(5,1)}{white}{white}{white}{white}
	\tgBorder{(5,1)}{0}{1}{0}{1}
	\tgBlank{(0,2)}{white}
	\tgBorderA{(1,2)}{white}{white}{white}{white}
	\tgBorder{(1,2)}{0}{1}{0}{1}
	\tgBorderA{(2,2)}{white}{white}{white}{white}
	\tgBorder{(2,2)}{1}{0}{1}{1}
	\tgBlank{(3,2)}{white}
	\tgBlank{(4,2)}{white}
	\tgBlank{(5,2)}{white}
	\tgBlank{(0,3)}{white}
	\tgBlank{(1,3)}{white}
	\tgBorderC{(2,3)}{0}{white}{white}
	\tgBorderA{(3,3)}{white}{white}{white}{white}
	\tgBorder{(3,3)}{0}{1}{0}{1}
	\tgBorderA{(4,3)}{white}{white}{white}{white}
	\tgBorder{(4,3)}{0}{1}{0}{1}
	\tgBorderA{(5,3)}{white}{white}{white}{white}
	\tgBorder{(5,3)}{0}{0}{0}{1}
	\tgBlank{(0,4)}{white}
	\tgBlank{(1,4)}{white}
	\tgBlank{(2,4)}{white}
	\tgBlank{(3,4)}{white}
	\tgBlank{(4,4)}{white}
	\tgBlank{(5,4)}{white}
	\tgCell[(1,1)]{(3.5,3)}{f}
	\tgCell[(0,1)]{(0.5,2)}{p}
	\tgDot{(5,3)}{}
	\tgAxisLabel{(6,1.5)}{west}{X}
\end{tangle}
\end{equation*}
and so by the naturality of the delete map:
\begin{equation*}
\begin{tangle}{(7,4)}
	\tgBlank{(0,0)}{white}
	\tgBlank{(1,0)}{white}
	\tgBlank{(2,0)}{white}
	\tgBlank{(3,0)}{white}
	\tgBlank{(4,0)}{white}
	\tgBlank{(5,0)}{white}
	\tgBlank{(6,0)}{white}
	\tgBlank{(0,1)}{white}
	\tgBlank{(1,1)}{white}
	\tgBlank{(2,1)}{white}
	\tgBorderC{(3,1)}{1}{white}{white}
	\tgBorderA{(4,1)}{white}{white}{white}{white}
	\tgBorder{(4,1)}{0}{1}{0}{1}
	\tgBorderA{(5,1)}{white}{white}{white}{white}
	\tgBorder{(5,1)}{0}{1}{0}{1}
	\tgBorderA{(6,1)}{white}{white}{white}{white}
	\tgBorder{(6,1)}{0}{1}{0}{1}
	\tgBorderA{(0,2)}{white}{white}{white}{white}
	\tgBorder{(0,2)}{0}{1}{0}{0}
	\tgBorderA{(1,2)}{white}{white}{white}{white}
	\tgBorder{(1,2)}{0}{1}{0}{1}
	\tgBorderA{(2,2)}{white}{white}{white}{white}
	\tgBorder{(2,2)}{0}{1}{0}{1}
	\tgBorderA{(3,2)}{white}{white}{white}{white}
	\tgBorder{(3,2)}{1}{0}{0}{1}
	\tgBlank{(4,2)}{white}
	\tgBlank{(5,2)}{white}
	\tgBlank{(6,2)}{white}
	\tgBlank{(0,3)}{white}
	\tgBlank{(1,3)}{white}
	\tgBlank{(2,3)}{white}
	\tgBlank{(3,3)}{white}
	\tgBlank{(4,3)}{white}
	\tgBlank{(5,3)}{white}
	\tgBlank{(6,3)}{white}
	\tgCell[(0,1)]{(0,2)}{p}
	\tgCell[(1,1)]{(1.5,2)}{f}
	\tgCell[(1,1)]{(4.5,1)}{h}
	\tgAxisLabel{(7,1.5)}{west}{X}
\end{tangle}
=
\begin{tangle}{(6,4)}
	\tgBlank{(0,0)}{white}
	\tgBlank{(1,0)}{white}
	\tgBlank{(2,0)}{white}
	\tgBlank{(3,0)}{white}
	\tgBlank{(4,0)}{white}
	\tgBlank{(5,0)}{white}
	\tgBlank{(0,1)}{white}
	\tgBlank{(1,1)}{white}
	\tgBorderC{(2,1)}{1}{white}{white}
	\tgBorderA{(3,1)}{white}{white}{white}{white}
	\tgBorder{(3,1)}{0}{1}{0}{1}
	\tgBorderA{(4,1)}{white}{white}{white}{white}
	\tgBorder{(4,1)}{0}{1}{0}{1}
	\tgBorderA{(5,1)}{white}{white}{white}{white}
	\tgBorder{(5,1)}{0}{1}{0}{1}
	\tgBlank{(0,2)}{white}
	\tgBorderA{(1,2)}{white}{white}{white}{white}
	\tgBorder{(1,2)}{0}{1}{0}{1}
	\tgBorderA{(2,2)}{white}{white}{white}{white}
	\tgBorder{(2,2)}{1}{0}{0}{1}
	\tgBlank{(3,2)}{white}
	\tgBlank{(4,2)}{white}
	\tgBlank{(5,2)}{white}
	\tgBlank{(0,3)}{white}
	\tgBlank{(1,3)}{white}
	\tgBlank{(2,3)}{white}
	\tgBlank{(3,3)}{white}
	\tgBlank{(4,3)}{white}
	\tgBlank{(5,3)}{white}
	\tgCell[(0,1)]{(0.5,2)}{p}
	\tgAxisLabel{(6,1.5)}{west}{X}
\end{tangle}.
\end{equation*}
So it is well defined to set $\tilde{\Psi}(h) \defeq \restr h {fp}$.
We can see that this is an inverse-with-support: to show this we must show that $\incl {\restr h {fp}}$ is an ordinary Bayesian inverse. i.e. that
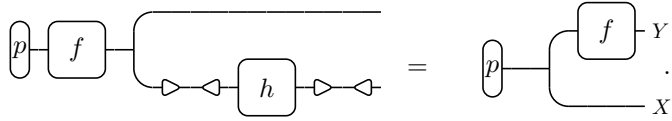
\begin{equation*}
\begin{tangle}{(10,4)}
	\tgBlank{(0,0)}{white}
	\tgBlank{(1,0)}{white}
	\tgBlank{(2,0)}{white}
	\tgBlank{(3,0)}{white}
	\tgBlank{(4,0)}{white}
	\tgBlank{(5,0)}{white}
	\tgBlank{(6,0)}{white}
	\tgBlank{(7,0)}{white}
	\tgBlank{(8,0)}{white}
	\tgBlank{(9,0)}{white}
	\tgBlank{(0,1)}{white}
	\tgBlank{(1,1)}{white}
	\tgBlank{(2,1)}{white}
	\tgBorderC{(3,1)}{1}{white}{white}
	\tgBorderA{(4,1)}{white}{white}{white}{white}
	\tgBorder{(4,1)}{0}{1}{0}{1}
	\tgBorderA{(5,1)}{white}{white}{white}{white}
	\tgBorder{(5,1)}{0}{1}{0}{1}
	\tgBorderA{(6,1)}{white}{white}{white}{white}
	\tgBorder{(6,1)}{0}{1}{0}{1}
	\tgBorderA{(7,1)}{white}{white}{white}{white}
	\tgBorder{(7,1)}{0}{1}{0}{1}
	\tgBorderA{(8,1)}{white}{white}{white}{white}
	\tgBorder{(8,1)}{0}{1}{0}{1}
	\tgBorderA{(9,1)}{white}{white}{white}{white}
	\tgBorder{(9,1)}{0}{1}{0}{1}
	\tgBorderA{(0,2)}{white}{white}{white}{white}
	\tgBorder{(0,2)}{0}{1}{0}{0}
	\tgBorderA{(1,2)}{white}{white}{white}{white}
	\tgBorder{(1,2)}{0}{1}{0}{1}
	\tgBorderA{(2,2)}{white}{white}{white}{white}
	\tgBorder{(2,2)}{0}{1}{0}{1}
	\tgBorderA{(3,2)}{white}{white}{white}{white}
	\tgBorder{(3,2)}{1}{0}{1}{1}
	\tgBlank{(4,2)}{white}
	\tgBlank{(5,2)}{white}
	\tgBlank{(6,2)}{white}
	\tgBlank{(7,2)}{white}
	\tgBlank{(8,2)}{white}
	\tgBlank{(9,2)}{white}
	\tgBlank{(0,3)}{white}
	\tgBlank{(1,3)}{white}
	\tgBlank{(2,3)}{white}
	\tgBorderC{(3,3)}{0}{white}{white}
	\tgBorderA{(4,3)}{white}{white}{white}{white}
	\tgBorder{(4,3)}{0}{1}{0}{1}
	\tgBorderA{(5,3)}{white}{white}{white}{white}
	\tgBorder{(5,3)}{0}{1}{0}{1}
	\tgBorderA{(6,3)}{white}{white}{white}{white}
	\tgBorder{(6,3)}{0}{1}{0}{1}
	\tgBorderA{(7,3)}{white}{white}{white}{white}
	\tgBorder{(7,3)}{0}{1}{0}{1}
	\tgBorderA{(8,3)}{white}{white}{white}{white}
	\tgBorder{(8,3)}{0}{1}{0}{1}
	\tgBorderA{(9,3)}{white}{white}{white}{white}
	\tgBorder{(9,3)}{0}{1}{0}{1}
	\tgCell[(1,1)]{(1.5,2)}{f}
	\tgCell[(0,1)]{(0,2)}{p}
	\tgRestr{(4,1)}{}
	\tgIncl{(5,1)}{}
	\tgCell[(1,1)]{(6.5,1)}{h}
	\tgRestr{(8,1)}{}
	\tgIncl{(9,1)}{}
\end{tangle}
=
\begin{tangle}{(5,5)}[trim y]
	\tgBlank{(0,0)}{white}
	\tgBlank{(1,0)}{white}
	\tgBlank{(2,0)}{white}
	\tgBlank{(3,0)}{white}
	\tgBlank{(4,0)}{white}
	\tgBlank{(0,1)}{white}
	\tgBlank{(1,1)}{white}
	\tgBorderC{(2,1)}{1}{white}{white}
	\tgBorderA{(3,1)}{white}{white}{white}{white}
	\tgBorder{(3,1)}{0}{1}{0}{1}
	\tgBorderA{(4,1)}{white}{white}{white}{white}
	\tgBorder{(4,1)}{0}{1}{0}{1}
	\tgBlank{(0,2)}{white}
	\tgBorderA{(1,2)}{white}{white}{white}{white}
	\tgBorder{(1,2)}{0}{1}{0}{1}
	\tgBorderA{(2,2)}{white}{white}{white}{white}
	\tgBorder{(2,2)}{1}{0}{1}{1}
	\tgBlank{(3,2)}{white}
	\tgBlank{(4,2)}{white}
	\tgBlank{(0,3)}{white}
	\tgBlank{(1,3)}{white}
	\tgBorderC{(2,3)}{0}{white}{white}
	\tgBorderA{(3,3)}{white}{white}{white}{white}
	\tgBorder{(3,3)}{0}{1}{0}{1}
	\tgBorderA{(4,3)}{white}{white}{white}{white}
	\tgBorder{(4,3)}{0}{1}{0}{1}
	\tgBlank{(0,4)}{white}
	\tgBlank{(1,4)}{white}
	\tgBlank{(2,4)}{white}
	\tgBlank{(3,4)}{white}
	\tgBlank{(4,4)}{white}
	\tgCell[(1,1)]{(3.5,3)}{f}
	\tgCell[(0,1)]{(0.5,2)}{p}
	\tgAxisLabel{(5,1.5)}{west}{X}
	\tgAxisLabel{(5,3.5)}{west}{Y}
\end{tangle}.
\end{equation*}
But this is follows straightforwardly by two applications of \cref{thm:restr-incl-almost-inv} and the fact that $h$ is a Bayesian inverse.
%

We finally have that $\Psi(-)$ is inverse to $\tilde\Psi$ when viewed as maps to/from equivalence classes:
\begin{alignat*}{3}
	\Psi(\tilde\Psi(h)) &=& i_p \circ r_p \circ h \circ i_{f p} \circ r_{f p}\\
	  &\aeq_{f p}& i_p \circ r_p \circ h\\
	  &\aeq_{f p}& h
\end{alignat*}	
where the final equivalence uses the fact that $h$ is a Bayesian inverse to move it out of the way, similarly to the previous chain of equalities.
\end{proof}

Noting that all Bayesian inverses to $f$ at $p$ must be $(f\circ p)$-almost equal we obtain \cref{thm:unique-inversion} as a corollary, that Bayesian inverses with support are unique.

\uniqueInversionThm*
\qed

\end{document}